\documentclass{amsart}
\usepackage{amsmath}
\usepackage{amssymb}
\usepackage{amsbsy}
\usepackage{amsfonts}
\usepackage{epsfig}
\usepackage{mathrsfs}
\usepackage{pict2e}
\usepackage{graphicx}
\usepackage{multirow}
\usepackage{epstopdf}
\usepackage{verbatim}
\usepackage{subfigure}
\usepackage{algorithm}
\usepackage{algpseudocode}
\usepackage{amsmath}
\usepackage{amssymb}
\newcommand{\argmin}[1]{\underset{#1}{\mathrm{argmin}}}

\usepackage{color}
\usepackage{varwidth}
\usepackage[normalem]{ulem}

\usepackage[table]{xcolor}

\usepackage{graphicx}
\usepackage{epstopdf}
\usepackage{float}
\usepackage{subfig}
\usepackage{epsfig}
\usepackage{fullpage}

\newtheorem{theorem}{Theorem}[section]

\newtheorem{lemma}[theorem]{Lemma}
\newtheorem{proposition}[theorem]{Proposition}
\newtheorem{corollary}[theorem]{Corollary}

\theoremstyle{definition}
\newtheorem{definition}[theorem]{Definition}

\theoremstyle{remark}
\newtheorem{remark}[theorem]{Remark}

\numberwithin{equation}{section}

\newcommand{\abs}[1]{\lvert#1\rvert}

\usepackage[T1]{fontenc}
\usepackage[utf8]{inputenc}
\usepackage[font=small,labelfont=bf]{caption}

\begin{document}

 \title{Exact Recovery of Chaotic Systems from Highly Corrupted Data}
\author{Giang Tran}
\address{Department of Mathematics, The University of Texas at Austin}
\email{gtran@math.utexas.edu}

\author{Rachel Ward}
\address{Department of Mathematics, The University of Texas at Austin}
\email{rward@math.utexas.edu}

\begin{abstract}
Learning the governing equations in dynamical systems from time-varying measurements is of great interest across different scientific fields.   This task becomes prohibitive when such data is moreover highly corrupted, for example, due to the recording mechanism failing over unknown intervals of time.  When the underlying system exhibits chaotic behavior, such as sensitivity to initial conditions, it is crucial to recover the governing equations \emph{with high precision}.  In this work, we consider continuous time dynamical systems $\dot{x} = f(x)$ where each component of $f: \mathbb{R}^{d} \rightarrow \mathbb{R}^d$ is a multivariate polynomial of maximal degree $p$; we aim to identify $f$ exactly from possibly highly corrupted measurements $x(t_1), x(t_2), \dots, x(t_m)$.  As our main theoretical result, we show that if the system is sufficiently ergodic that this data satisfies a strong central limit theorem (as is known to hold for chaotic Lorenz systems), then the governing equations $f$ can be \emph{exactly recovered} as the solution to an $\ell_1$ minimization problem -- \emph{even if a large percentage of the data is corrupted by outliers}.  Numerically, we apply the alternating minimization method to solve the corresponding constrained optimization problem.  Through several examples of 3D chaotic systems and higher dimensional hyperchaotic systems, we illustrate the power, generality, and efficiency of the algorithm for recovering governing equations from noisy and highly corrupted measurement data.  
\end{abstract}

\maketitle

%=============================================================================
%=============================================================================
%=============================================================================
\section{Introduction}
Discovering the underlying dynamical equations from time-dependent observations is of great interest across many scientific fields. Examples include statistical learning theory \cite{vapnik2013nature}, manifold learning \cite{roweis2000nonlinear}, machine learning \cite{jordan2015machine}, physical modeling \cite{schittkowski2013numerical}, and system identification \cite{ljung1998system,ljung2010perspectives}. In those examples, it is generally assumed that the governing equations can be expressed as a combination of terms in an appropriate functional space \cite{hastie2015statistical}. The selection of the important terms within this space is obtained from regularization, pruning, shrinking or regression.  An overview of reconstruction methods for dynamical systems can be found in \cite{sjoberg1995nonlinear}. Indeed, it is shown that without additional information, learning the governing equations from measurement data is \textit{intractable} \cite{cubitt2012extracting}, and may face the \textit{curse of dimensionality} \cite{fornasier2012learning}, regardless of how much data there is.

Identification of nonlinear dynamical systems is one of the most active areas in system identification and one of the main topics in the rapid development of nonlinear dynamics \cite{ljung2010perspectives, aguirre2009modeling}. It is shown that simple nonlinearities in the governing equations can lead to incredibly complicated behavior in the solutions of the nonlinear systems, which is the so-called \emph{butterfly effect} present in chaotic solutions, \cite{sprott1994some,sprott2000algebraically,lu2004new}. For example, the well-known three-variable Lorenz system \cite{lorenz1963deterministic} has seven terms with only one nonlinear term of quadratic type on the right hand side, and the R{\" o}ssler system \cite{rossler1976chemical} has six terms with only one nonlinear term, yet their solutions exhibit chaotic behavior.

 There have been many approaches to extract the underlying structures of chaotic systems from time-dependent data. A review of major methods in modeling nonlinear dynamics and chaotic systems can be found in\cite{aguirre2009modeling}. One of the main directions is the reconstruction of state space from one-dimensional realizations, \cite{packard1980geometry, crutchfield1987equations, casdagli1991state, kennel1992determining, kugiumtzis1996state}. Those methods rely on the estimation of the time delay and the embedding dimension to reconstruct a state space which preserves the topological properties of the original system. In \cite{crutchfield1987equations, rowlands1992extraction}, the authors use the singular value decomposition to determine the appropriate dependent variables that will appear in the dynamical equations. The reconstruction guarantee in the absence of noise, called the \textit{delay embedding theorem}, is proved by \cite{takens1981detecting}. However, it is showed that in some special choices of parameters,  two different systems can produce the exact same time series of one of their variables \cite{lainscsek2012class}.
 
 Even with the availability of state-space data, it is challenging to recover the parameters in the governing equations.   In \cite{schmidt2009distilling}, the authors use symbolic regression to find both the parameters and the forms of the equations simultaneously.  In \cite{brunton2015discovering}, the authors recast the problem of recovering coefficients in a known basis as a linear regression, where the matrix for regression is built from the data.  They moreover incorporate thresholding in the regression to further promote sparsity in the recovered coefficients that govern the chaotic systems.

Sparsity has been playing a significant role in the developments of compressed sensing, image processing, optimization, and many others. Recently, sparse-inducing methods used in image processing and compressed sensing have been applied to partial differential equations, dynamical systems and physical sciences \cite{williams2012low, schaeffer2013sparse, cheng2013dynamically, bright2013compressive,ozolicnvs2013compressed,ozolicnvs2014compressed, proctor2014exploiting, mackey2014compressive, brunton2014compressive, caflisch2015pdes,  tran20151,  brunton2015discovering}. In these works, the authors study either the sparse property of the solutions in different contexts or the sparse structures of the governing equations. In the latter direction, it turns out that many chaotic systems have simple algebraic representations corresponding to a sparse representation in high dimensional nonlinear functional spaces \cite{sprott1994some,sprott2000algebraically,lu2004new}.

In this work, we bring together connections between compressed sensing, splitting optimization methods, sparse representations of the governing equations, and the statistical properties of chaotic systems, to provide \emph{exact recovery guarantees} for classes of chaotic systems. In particular, we provide conditions and algorithms for recovering the governing equations from possibly highly corrupted data generated from a class of \emph{Lorenz-like systems} which includes the classical Lorenz equations, and are known to be ergodic.  Explicitly \textbf{when the underlying attractor of the flow has Hausdorff dimension greater than two, the flow satisfies some mixing properties and the governing equation vector $f$ has a sparse representations in the space of multivariable polynomials,  then the polynomial coefficients of $f$ as well as the outlier vectors can be exactly recovered as the unique solution to a partial $\ell_1$-minimization problem with high probability (depends on the number of measurements), as long as the number of measurements is big enough and sparse level is low enough.} We prove theoretical reconstruction guarantees by combining the partial sparse recovery results \cite{bandeira2013partial} with statistical behavior of the Lorenz-like systems \cite{araujo2014statistical,araujo2015rapid}. It is based on the observation that although  individual trajectories are highly unpredictable due to the sensitivity property to initial conditions of chaotic dynamic systems, their statistical behavior is understandable and share many of the same properties of random sequences.  We finally mention that for our theoretical results, we do not necessarily require the system to be chaotic in the formal sense, but rather we use certain ergodicity properties which are satisfied for a class of geometric Lorenz attractors which are also known to be chaotic.   Therefore, our theory can likely be generalized to other dynamical systems with similar ergodicity properties.

The paper is divided as follows. In Section 2, we explain the problem setting. In Section 3, we first review some results from compressed sensing and statistical properties of chaotic systems. Then we present our theoretical reconstruction guarantees and state in which conditions the sparse solutions can be recovered. The numerical implementations and results are described in Sections 4 and 5. Concluding remarks are given in Section 6.

%=============================================================================
%=============================================================================
%=============================================================================
\section{Problem Setting}
In this work, we are interested in reconstructing the governing equations of a chaotic system from the time-varying measurement data in which data is corrupted at unknown intervals of time. Here, we consider  chaotic systems of the form
\begin{equation}
\dfrac{d}{dt}x(t) = f(x(t)),
\end{equation}
where the column vector $x(t) = (x_1(t), x_2(t),\ldots, x_d(t))^T$ represents the state of the system at time $t$ and the nonlinear function vector $f(x) = (f_1(x),f_2(x),\ldots, f_d(x))^T$ defines the dynamic motions. Even though chaotic systems are deterministic, they exhibit stochastic behavior \cite{lorenz1963deterministic}. Such behavior is commonly referred as the \textit{butterfly effect}, \textit{i.e.}, small differences in initial conditions will yield much larger differences in the outcomes. Therefore, it is very important to recover the governing equations with high accuracy.

In many chaotic systems inspired by physical and biological processes, the governing equations may consist of only a few terms in a high dimensional nonlinear functional space \cite{lorenz1963deterministic, rossler1976chemical,sprott1994some,sprott2000algebraically,lu2004new}. This observation also holds for many dynamical systems and partial differential equations where the complicated system can be well-approximated by a simpler system with similar behavior \cite{schaeffer2013sparse, brunton2015discovering}. Explicitly, for many chaotic systems, the governing equations can be represented in the space of polynomial functions
\begin{equation}
\label{eq:poly}
f_j(x(t)) = c_{j0} + \sum\limits_{k} c_{j,k}x_k(t) + \sum\limits_{k,l} c_{j,k,l}\, x_k(t) \, x_l(t) +\sum\limits_{k,l,n} c_{j,k,l,n}\, x_k(t) \, x_l(t)\, x_n(t) + \ldots
\end{equation}
where the coefficient vectors $c_{j} = (c_{j,0} ,c_{j,1},\ldots, c_{j,d}, c_{j,1,1},\ldots, c_{j,d,d},\ldots)$ are moreover sparse.  Note that $c_j$ is a vector of length 
\begin{equation}
\label{eq:r}
r = {p+d \choose d} \leq \left( \frac{p+d}{d} \right)^d,
\end{equation}
the maximal number of monomials of degree at most $p$ in a multivariate polynomial in $d$ variables.  In fact, many key dynamical systems  arising from applications in biology and physics, such as the Lorenz system, are in fact \emph{bilinear}, so that $p=2$.

For the rest of the paper, we denote $d$ the dimension of the system, $m$ the number of measurements and $r$ the cardinal of the basis for a given nonlinear functional space. Given a collection of data at different times $\{x(t_1), x(t_2),\ldots, x(t_m)\}$, as in \cite{brunton2015discovering}, we construct the following matrices $X$, $\dot{X}$ and $\Phi(X)$ to store the measurement data, the time derivative of the data and the dictionary built from the data:

\begin{align}
   X &= \begin{bmatrix}
       |  &   |     &           & |      \\
   X_1 & X_2 & \ldots & X_d  \\
       |  &  |      &            & |      \\
   \end{bmatrix}
   = \begin{bmatrix}
           x_1(t_1) & x_2(t_1) &\ldots & x_d(t_1)\\
           x_1(t_2) & x_2(t_2)&\ldots & x_d(t_2)\\
           \vdots &\vdots&\ldots & \vdots\\
           x_1(t_m) & x_2(t_m) &\ldots & x_d(t_m)\\
         \end{bmatrix}_{m\times d} \\
       \dot{X} &= \begin{bmatrix}
        |          &   |            &           & |      \\
   \dot{X}_1 & \dot{X}_2 & \ldots & \dot{X}_d  \\
       |           &  |                &            & |      \\
   \end{bmatrix}
_{m\times d}
 = \begin{bmatrix}
           \dot{x}_1(t_1) & \dot{x}_2(t_1) &\ldots & \dot{x}_d(t_1)\\
           \dot{x}_1(t_2) & \dot{x}_2(t_2)&\ldots & \dot{x}_d(t_2)\\
           \vdots &\vdots&\ldots & \vdots\\
           \dot{x}_1(t_m) & \dot{x}_2(t_m) &\ldots & \dot{x}_d(t_m)\\
         \end{bmatrix}_{m\times d} \\
          \Phi(X)  &= \begin{bmatrix}
           |  &  | & |      & |  &    \\
           1 & X & X^{P_2} & X^{P_3} & \ldots \\
           |  &  | & |      & |  &    \\
            \end{bmatrix}_{m\times r}
  \end{align}
 where for each positive integer $k$, $X^{P_k}$ denotes the values of all monomials of total degree $k$ at the times $t_1, t_2,\ldots, t_m$. For example, 
\begin{equation}
    \begin{aligned}
   X^{P_2} &= \begin{bmatrix}
           x_1^2(t_1) & x_1(t_1) x_2(t_1) &\ldots & x_1(t_1) x_d(t_1) & x_2^2(t_1) &\ldots & x_d^2(t_1)\\
           x_1^2(t_2)& x_1(t_2) x_2(t_2) & \ldots& x_1(t_2) x_d(t_2) & x_2^2(t_2) & \ldots & x_d^2(t_2) \\
           \vdots & \vdots & \ldots &\vdots &\ldots &\vdots\\
           x_1^2(t_m) & x_1(t_m) x_2(t_m) & \ldots & x_1(t_m) x_d(t_m)&x_2^2(t_m) & \ldots & x_d^2(t_m)
         \end{bmatrix}_{m\times \left(\frac{d(d+1)}{2}\right)}
         \end{aligned}
         \end{equation}
In the case of non-damaged data, \cite{brunton2015discovering} observed that the problem of finding $f$ of the form \eqref{eq:poly}  in the case of exact data $X$ and $\dot{X}$ can be reformulated as finding the coefficient matrix ${\mathcal C}=[c_1 \quad c_2\quad \ldots\quad c_d]_{r\times d}$ such that $\dot{X} = \Phi(X) {\mathcal C}$.  Thus, given noisy data, such as if $\dot{X}$ is approximated via finite differences of successive values of $X$, they propose to solve for ${\mathcal C}$ via linear regression: 
\begin{equation}
{\mathcal C} = \argmin{\widetilde{\mathcal C}} \dfrac{1}{2}\| \dot{X} - \Phi(X) \widetilde{\mathcal C} \|_2^2
\end{equation}
The authors moreover incorporate hard thresholding into the regression procedure to promote sparsity in ${\mathcal C}$. They illustrate the effectiveness of their algorithm for stably reconstructing the governing coefficients in the presence of a small amount of noise, but do not provide theoretical performance guarantees.  In particular, there are no guarantees a priori that the matrix $\Phi(X)$ will be full-rank, which is necessary for uniqueness of the recovered ${\mathcal C}$. For example, the matrix $\Phi(X)$ will not have full rank if $x(t_k) = x(t_{k+1})$ is at a fixed point, or if  $x(t_{k+L}) = x(t_{k})$ is fixed at a cycle of length $L < r$ more generally.

As a by-product of our main theorem, we provide conditions under which the regression algorithm from \cite{brunton2015discovering} is theoretically justified; namely, will show that \emph{for chaotic data,  the columns are independent, and thus the matrix $\Phi(X)$ has full column rank and thus the polynomial coefficients generating the data is unique.}  See Corollary \ref{cor:fullrank} for more details.  More generally, we are interested in the case when the data is might be highly corrupted, such as over unknown intervals of time due to disruption of the measurement device.  

Our set-up is as follows: consider data $X^\circ$ which is corrupted in the sense that some small fraction of the $m$ measurements $x(t_1), x(t_2), \dots, x(t_m)$ and $\dot{x}(t_1), \dot{x}(t_2), \dots, \dot{x}(t_m)$ are perturbed by bounded additive error: $x^{o}(t_j) = x(t_j) + \theta_j$ and $\dot{x}^{o}(t_j) = \dot{x}(t_j) + \theta'_j$.   Exploiting that the corruptions are sparse, the matrix difference ${\mathcal E} = \dot{X^\circ} - \Phi(X^\circ){\mathcal C}$ will only have a small fraction of its rows which are non-zero.  Thus, denoting the $j$th row of ${\mathcal E}$ by $\mathcal{E}(j,:)$, the optimization problem of jointly recovering the polynomial governing coefficients and locations of the corruptions can be posed as a sparse recovery problem:
\begin{equation}
\begin{aligned}
\min\limits_{(\mathcal C,\mathcal E)} \,  &\|  \{ j:   \mathcal E(j,:) \text{ is non-zero } \} \| \\
&\text{  subject to}\quad \Phi(X^\circ){\mathcal C}+ {\mathcal E}=\dot{X^\circ}. 
\end{aligned}
\end{equation}
Of course, this optimization problem is intractable, so we relax the objective function to be convex.  To enforce the group-sparsity with respect to the rows of ${\mathcal E}$ \cite{fornasier2008recovery, kowalski2009sparse, deng2013group}, we consider the following optimization model
\begin{equation}
\begin{aligned}
\min\limits_{({\mathcal C}, {\mathcal E})} \, \| {\mathcal E}\|_{2,1} = &\min\limits_{({\mathcal C}, {\mathcal E})} \, \sum\limits_{j=1}^m\, \| {\mathcal E}(j,:) \|_2,\\
&\text{subject to}\quad \Phi(X^\circ){\mathcal C} + {\mathcal E}  =\dot{X^\circ}. 
\end{aligned}
\label{model:main}
\end{equation}
For more details, see Section 3. To enforce additional sparsity in the polynomial coefficient matrix ${\mathcal C}$, we will also consider the problem
\begin{equation}
\begin{aligned}
\min\limits_{({\mathcal C}, {\mathcal E})} \, \| {\mathcal E} \|_{2,1} = &\min\limits_{({\mathcal C}, {\mathcal E})} \, \sum\limits_{j=1}^m\, \| {\mathcal E}(j,:) \|_2,\\
&\text{subject to}\quad \Phi(X^\circ){\mathcal C} + {\mathcal E}  =\dot{X^\circ} \quad \text{and}\quad {\mathcal C} \quad \text{is sparse.}
\end{aligned}
\label{model:sparse}
\end{equation}
For more details, see Sections 4 and 5.

\section{Reconstruction Guarantee Analysis}
In this section, we first recall some results from compressive sensing and in particular, partial sparse recovery problems therein, as well as statistical properties of Lorenz-like systems. Then we present our theoretical guarantees for the framework \eqref{model:main}.

\subsection{Theory from Compressive Sensing}

The compressive sensing paradigm, in its most basic form as introduced in  \cite{do06b, candes2006stable}, considers, for an underdetermined linear system of equations $y = A x$, the NP hard minimization problem
\begin{equation}
\begin{aligned}
x_0 = &\argmin{z} \quad  \| z \|_0 = | \{ j: | z_j| > 0 \} |  \\
&\text{  subject to}\quad Az = y.
\end{aligned}
\label{l0}
\end{equation} 
The convex relaxation of this problem is the $\ell_1$ minimization problem
\begin{equation}
\begin{aligned}
x_1 = &\argmin{z} \quad \| z \|_1 = \sum_{j=1}^n | z_j | \\
&\text{  subject to}\quad Az = y.
\end{aligned}
\label{l1}
\end{equation} 

Compressive sensing theory provides conditions on the underdetermined matrix $A$ such that the solutions $x_0$ and $x_1$ are equivalent and equal to $x$ satisfying $Ax = y$ whenever there exists such an $x$ is sufficiently sparse.  The theoretical guarantees are also stable with respect to non-exact sparse solutions and robust with respect to additive noise on the measurements $y = Ax$, but for simplicity we discuss only the case of exact sparsity here.  For a comprehensive overview of compressive sensing, we refer the reader to \cite{foucart2013mathematical}.

From here on out, we say that a vector $x \in \mathbb{R}^n$ is \emph{$s$-sparse}.  As shown in \cite{do06b, candes2006stable, cohen2009compressed}, a certain  \emph{null-space property} for an $m \times n$ matrix $A$ is a sufficient and necessary condition for sparse solutions to be exactly recovered via $\ell_1$ minimization. 

\begin{proposition}
\label{prop:nsp}
Given a matrix $A \in \mathbb{R}^{m \times n}$, every $s$-sparse vector $x$ is the unique solution of \eqref{l1} with $y = Ax$ if and only if for every $v \in \mathbb{R}^n \setminus \{0\}$ in the null space of $A$, and for every set $S \subset \{1 , 2, \dots, n\}$ of cardinality $s$, the following holds:
$$
\| v_S \|_1 < \frac{1}{2} \| v \|_2.
$$
\end{proposition}

In words, the null space property means that all vectors in the null space of the measurement matrix $A$ should be sufficiently un-concentrated on any subset of its entries, or rather, sufficiently ``far" from the nonlinear set of sparse vectors. 

More recently, a variant of compressive sensing theory has been developed for the theory of \emph{partial sparse} recovery.  Although there is a rich literature on partial sparse recovery guarantees, we will use the results from \cite{bandeira2013partial} which are most closely related to the setting at hand.

We use the following definition and theorem for \emph{partial sparse recovery} from \cite{bandeira2013partial}:

\begin{definition}[Partial null space property]
A pair of matrices $A = (A_1^{l \times (n-r)}, A_2^{l\times r})$ satisfies the \emph{null space property} (NSP) of order $s-r$ for partially sparse recovery of size $n-r$ with $r \leq s$ if $A_2$ is \emph{full column rank} and if for every $v_1 \in \mathbb{R}^{n-r}\backslash \{0\}$ such that $A_1 v_1 \in {\mathcal R}(A_2),$ the range of $A_2$, and for every set $S \subset \{1,\ldots, n-r\}$ of cardinality $s-r$, the following holds
$$
\| (v_1)_S \|_1 < \frac{1}{2} \| A_1 v_1 \|_1.
$$
\end{definition}

\begin{proposition}
\label{prop_big}
A pair of matrices $A = (A_1^{l\times (n-r)}, A_2^{l\times r})$ satisfies the NSP of order $s-r$ for partially sparse recovery of size $n-r$ if and only if every $\bar{v} = (\bar{v}_1, \bar{v}_2)$, such that $\bar{v}_1 \in \mathbb{R}^{n-r}$ is $(s-r)$-sparse and $\bar{v}_2 \in \mathbb{R}^r$ such that $A_1 \bar{v}_1 + A_2 \bar{v}_2 = y$, is the unique solution to 
\begin{equation}
\label{l1partial}
\min\limits_{v'_1, v'_2} \hspace{1mm} \| v'_1 \|_1 \hspace{1mm} \emph{subject to} \hspace{1mm} A_1 v'_1 + A_2 v'_2 = y. \nonumber
\end{equation}
\end{proposition}

We will use this proposition in the particular case $A_1 = Id_{l \times l}.$   We state this special case explicitly for clarity.

\begin{corollary}
\label{prop_cs}
Every $(v, w) \in \mathbb{R}^{m + r}$ satisfying $v + A w = y$ such that $v \in \mathbb{R}^m$ is $(s-r)$-sparse is the unique solution to 
\begin{equation}
\label{l1partial+}
\min\limits_{\tilde v, \tilde w} \hspace{1mm} \| \tilde v \|_1 \hspace{1mm} \emph{subject to} \hspace{1mm} \tilde v + A \tilde w = y, \nonumber
\end{equation}
if and only if $A \in \mathbb{R}^{m \times r}$ is full column rank and for every $\tilde v \in \mathbb{R}^{m}\backslash \{0\}$ such that $\tilde v \in {\mathcal R}(A),$ the following holds  for every set $S \subset \{1,\ldots, m\}$ of cardinality $s-r$:
$$
\| \tilde v_S \|_1 < \frac{1}{2} \| \tilde v \|_1.
$$
\end{corollary}

Note that a straightforward corollary of this corollary is that exact and unique recovery still holds if we add any additional consistent linear constraints to the linear program. 
\begin{corollary}
\label{cor_cs}
Under the same conditions as above, if for some pair of matrices $(B_1, B_2)$ the solution vector $(v,w)$ satisfies additionally  $B_1 v + B_2 w = z,$ then under the same conditions $(v,w)$ is the unique minimizer to the program
\begin{equation}
\label{l1partial++}
\min\limits_{\tilde v,\tilde w} \hspace{1mm} \| \tilde v \|_1 \hspace{1mm}\emph{subject to} \hspace{1mm} \tilde v + A \tilde w = y, \hspace{1mm} B_1 \tilde v + B_2 \tilde w = z. \nonumber
\end{equation}
\end{corollary}

\subsection{Statistical Behavior of Lorenz-like Systems}
In this section, we are interested in 3D dynamical systems, i.e., the dimension of the system is $d=3$. We first recall the following notation.  We denote by $C^{1+\eta}(\Omega)$ the Holder space consisting of those functions $f$ having continuous derivative up to order 1 and such that all partial derivatives are Holder continuous with exponent $\eta$: 
\[
\| f \|_{C^{1+\eta}(\Omega)} := \max_{| \beta| \in \{0,1\} } \sup_{x \in \Omega} | D^{\beta} f(x) | +  \max_{|\beta| = 1} \sup_{x \neq y \in \Omega}  \frac{ | D^{\beta} f(x) - D^{\beta} f(y)|}{|x-y|^{\eta} },
\]
 where $\beta$ ranges over multi-indices with $|\beta| = \sum\limits_i |\beta_i|$.

Recall the well-known classical \emph{Lorenz equations}: 
\begin{eqnarray}
\label{eq:lorenz}
\left. \begin{array}{l}
\dot{x_1} = a\, (x_2 - x_1) \\
\dot{x_2} = \gamma x_1 - x_2 - x_1x_3  \\
\dot{x_3} = x_1x_2 - b\, x_3, \\
\end{array} \right.
\quad 
\left. \begin{array}{l}
a = 10 \\
\gamma = 28 \\
b = 8/3  \\
\end{array} \right.
\end{eqnarray}
Note that the Lorenz equations are of the form $\dot{x}(t) = f(x(t))$ with polynomial governing equations $f$ of the form \eqref{eq:poly} of degree $p=2$.   With slight abuse of notation, we will go back and forth also between the notation $\dot{X^t} = f(X^t)$.  

The author in \cite{lorenz1963deterministic} introduced these equations as a simplified model for weather forecast, and numerical simulations indicated that in an open neighborhood of the chosen parameters, almost all points in phase space tend to a chaotic attractor.  One property of chaotic systems is \emph{sensitive dependence on initial conditions}, which implies that long term predictions based on such models are infeasible; on the other hand, another property of chaos is that the \emph{statistical} behavior of such systems is understandable, and chaotic systems share many of the same statistical properties of random sequences. The Lorenz attractor, while easy to visualize numerically, has proved extremely difficult to analyze rigorously.  A proof of existence of the Lorenz attractor was only provided only in 1999 \cite{tucker1999lorenz}, incorporating a computer-aided proof. Precisely, Tucker proved that the Lorenz equations \eqref{eq:lorenz} support a compact, connected attractor $\Lambda$ and the flow admits a unique so-called ``physical" measure $\mu$ with $\text{supp}(\mu) = \Lambda$.  An invariant probability measure $\mu$ for a flow $X^t:=(x_1(t), x_2(t),x_3(t))$ on a compact Riemannian manifold $M$ is called physical if the \emph{basin} of $\mu$, $B(\mu)$ has positive Lebesgue measure.  Recall that $B(\mu)$ is the set of points $z \in M$ satisfying for all continuous functions $\psi: M \rightarrow \mathbb{R}$
\begin{equation}
\label{physical}
\lim_{T \rightarrow \infty} \frac{1}{T} \int_{0}^T \psi(X^t(z)) dt = \int_{\Lambda} \psi(z) d\mu(z).
\end{equation}
Roughly speaking, the existence of a physical measure for an attractor means that most points in a neighborhood of the attractor have well defined long term statistical behavior.  For the Lorenz equations \eqref{eq:lorenz}, and more generally, for any so-called \emph{geometric Lorenz-like system} (see \cite{araujo2014statistical} for definition and properties) of a flow on a three-dimensional manifold, the ergodic basin $B(\mu)$ covers a full Lebesgue measure subset of the topological basin of attraction $\Lambda$. 

Property \eqref{physical} of a physical measure shows that asymptotically, the time average of a continuous observable of the flow equals its space average.  It is natural to ask more quantitatively for the \emph{rate of convergence} of the time averages to the space average and moreover, if such ergodicity also holds for the time-1 map for the flow, given by the discrete sequence $\{ X^j \}_{j \in \mathbb{Z}}$.  Following several results in this direction for certain classes of so-called geometric Lorenz attractors which include the Lorenz equations \eqref{eq:lorenz}, we state one of the most recent results in this direction.   Combining Theorems 5.2 and Theorem 7.1 from \cite{araujo2015rapid} with Corollary 2.2 of \cite{araujo2015exponential}, one can derive estimates on the rate of mixing in \eqref{physical} for the discrete time-1 map  in the form of an \emph{almost sure invariance principle} (ASIP) for the time-1 map of flows generated by the Lorenz equations \ref{eq:lorenz} and, more generally, for the time-1 map of a vector field belonging to the following (quite technical) class: 

\begin{definition}
\label{C0}
 Denote by ${\mathcal U}$ the class of $C^{1+\eta}$ uniformly hyperbolic skew product flows, subject to a uniform nonintegrability condition, as defined by the condition UNI and properties (i)-(iv) from \cite{araujo2015exponential}.   This class includes the classical Lorenz attractor, and an open set of coefficients around the classical coefficients.
 \end{definition}

\begin{proposition}[ASIP for time-1 maps]
\label{CLT_main}
Fix $\eta > 0$.  Let $X^t$ be the flow generated by a vector field $G \in {\mathcal U}$ starting from $X^0 = x \in \Lambda$, and consider its time-1 map 
$ \{X^0,X^1, X^2, \ldots,X^m \}$.  Let $\psi: \mathbb{R}^3 \rightarrow \mathbb{R}$ be a $C^{1+\eta}$ function, and let $Z$ be a standard normal random variable.  Then there is a universal constant $C_1 > 0$ and a constant $C_{x,  \psi} \geq 0$ such that
\begin{align}
 \left| \frac{1}{m} \sum_{j=0}^{m-1} \psi(X^j) - \int_{\Lambda} \psi(z) \hspace{.5mm} d\mu(z) - \frac{\sigma}{\sqrt{m}} Z  \right| &\leq  C_{x, \psi} m^{-3/4} (\log(m))^{1/2} (\log \log(m))^{1/4},  \nonumber \\
 &\quad \text{for } \mu \text{-almost all }  x \in \Lambda, \nonumber
\end{align}
and the variance is bounded by $\sigma^2 \leq C_1 \| \psi \|^2_{C^{1+\eta}(\Lambda)}$.
\end{proposition}

\bigskip

\begin{remark}
Note that the ASIP \emph{implies} the Central Limit Theorem (CLT): in the same setting as above, 
$$
 \frac{1}{\sqrt{m}} \left( \sum_{j=0}^{m-1} \psi( X^j)-  m \int_{\Lambda} \psi(x) d\mu(x)  \right)  \longrightarrow 
{\mathcal N}(0, \sigma^2) \quad\text{as}\quad m\rightarrow\infty,
$$ 
where the convergence is in distribution.
\end{remark}

For our purposes, it will be useful to state a more uniform version of Proposition \ref{CLT_main}.  
Consider the function
\begin{equation}
\label{uniform_asip}
F_{\tau, \eta}(x) = \sup_{\psi: \| \psi \|_{C^{1+\eta}} \leq \tau} C_{x, \psi}.
\end{equation}
By Proposition \ref{CLT_main}, $F_{\tau}(x)$ is finite for $\mu$-almost all $x$. 
Thus, given any $\varepsilon > 0$, there exists a constant $\kappa_{\varepsilon, \tau, \eta}$ and a subset $\Lambda_{\varepsilon, \tau, \eta} \subset \Lambda$ of measure $\mu(\Lambda_{\varepsilon, \tau, \eta}) \geq 1 - \varepsilon$ such that $F_{\tau, \eta}(x) \leq \kappa_{\varepsilon, \tau, \eta}$ uniformly for all $x \in \Lambda_{\varepsilon, \tau, \eta}$. In terms of this constant, we can state the following corollary.

\begin{corollary}
\label{CLT_main2}
Fix $\eta > 0$ and $\varepsilon > 0$.   Let $X^t$ be the flow generated by a vector field $G \in {\mathcal U}$.  Draw $x \in \Lambda$ from the measure $d\mu$, and consider the flow $X^t$ generated by such a vector field originating at $X^0 = x$, and its time-1 map 
$ \{X^0, X^1, X^2, \dots \}$.   There is a universal constant $C_1 > 0$ and a constant $\kappa_{\varepsilon, \tau, \eta} \geq 0$ such that with probability exceeding $1 - \varepsilon$ with respect to the draw of $x$, the following holds uniformly over all $\psi: \mathbb{R}^3 \rightarrow \mathbb{R}$ satisfying $\| \psi \|_{C^{1+\eta}} \leq \tau$:  
\begin{align}
 \left| \frac{1}{m} \sum_{j=0}^{m-1} \psi(X^j) - \int_{\Lambda} \psi(z) \hspace{.5mm} d\mu(z) -  \frac{\sigma}{\sqrt{m}} Z \right| &\leq \kappa_{\varepsilon, \tau, \eta}  m^{-3/4} (\log(m))^{1/2} (\log \log(m))^{1/4}, \nonumber 
\end{align}
and the variance is bounded by $\sigma^2 \leq C_1 \tau^2$.
\end{corollary}

For any $G \in {\mathcal U}$, the corresponding attractor $\Lambda$ will be compact, \textit{i.e.}, there exists some finite $B_{\Lambda} > 0$ such that 
$$\max_{(x,y,z) \in \Lambda} \{ | x |, |y|, |z| \} \leq B_{\Lambda}.$$
For the Lorenz equations \eqref{eq:lorenz} in particular, $\Lambda$ is a fractal, and has measured Hausdorff dimension $2.06 \pm .01$ \cite{viswanath2004fractal}.

\subsection{Recovering Polynomial Dynamics from Chaotic Data}
Fix number of measurements $m$. Fix $\Theta$ and $\Theta' \in \mathbb{R}^{m \times 3}$, arrays of sparse corruptions such that 
$\| \Theta(:,l) \|_0, \| \Theta'(:,l)  \|_0 \leq s$, and are uniformly bounded,
\begin{equation}
\label{noise_bound}
\sup_j  \left\{ \| \Theta(j,l) \|_{\infty},\| \Theta'(j,l) \|_{\infty} \right\}  \leq B_{\Theta}, \quad \quad l = 1,2,3.
\end{equation}
Suppose that we observe corrupted iterations of the time-1 map of a flow $X^t = (x_1(t), x_2(t), x_3(t))$ satisfying the conditions of Corollary \ref{CLT_main2}:
\begin{equation}
\label{observe}
\text{Given:} \quad U^t = X^t + \Theta_t, \quad V^t =  \dot{X}^t + \Theta'_t,  \quad \quad  t = 0, 1, 2, \dots.
\end{equation}

Our measurements $U^j = (u_1(j) ,u_2(j), u_3(j))$ and $V^j = (v_1(j), v_2(j), v_3(j))$ satisfy the linear equations 
\begin{equation}
\label{recast}
V- \Phi {\mathcal C}  = {\mathcal E},
\end{equation}
where 
\begin{itemize}
\item $\Phi$ has rows $\Phi(j,:) = (1, u_1(j), u_2(j), u_3(j), u_1(j) u_2(j), \dots )$
\item ${\mathcal E} = (e_1, e_2, e_3) \in \mathbb{R}^{m \times 3}$ has a sparse number of nonzero rows, with $\| {\mathcal E} \|_0 \leq 6s$,
\item ${\mathcal C} = (c_1, c_2, c_3 ) \in \mathbb{R}^{r \times 3}$ is the matrix of polynomial coefficients.
\end{itemize}

\bigskip
Notice that by turning the matrices $V, {\mathcal C}, {\mathcal E}$ into tall column vectors $v, c, e$ and turning $\Phi$ into the augmented matrix  
$$
A=\left[ \begin{array} {ccc}
\Phi, 0, 0 \\
0, \Phi, 0 \\
0, 0, \Phi
\end{array} \right],
$$ 
we can equivalently write equation \eqref{recast} as $v- A c  = e$.
The range space ${\mathcal R}(A)$ corresponds to vectors $ A c = (\Phi c_1, \Phi c_2, \Phi c_3) \in \mathbb{R}^{3m\times 3}$ of the form
\begin{eqnarray}
\label{range}
( \Phi c_l )_j &=& \sum_{\alpha: | \alpha | \leq p} c_l(\alpha)  u_1(j)^{\alpha_1} u_2(j)^{\alpha_2} u_3(j)^{\alpha_3} \nonumber \\
&=& \sum_{\alpha: | \alpha | \leq p} c_l(\alpha) x_1(j)^{\alpha_1} x_2(j)^{\alpha_2} x_3(j)^{\alpha_3} + R_l(j), \quad \quad l = 1,2,3,
\end{eqnarray}
where $\| R_l \|_0 \leq 2s$ and, in light of the assumption \eqref{noise_bound}, is bounded by
\begin{equation}
\label{bound_residual}
|R_l(j)| \leq  (B_{\Theta}+B_{\Lambda})^d \| c_l \|_1, \quad \quad l = 1,2,3.
\end{equation}

\begin{lemma}
\label{lemma_lower}
Let $\eta \geq 0$.  Suppose that the underlying attractor $\Lambda$ for the flow at hand has Hausdorff dimension greater than two, and consider the function $\psi^c = \psi^{(c_1, c_2, c_3)}: \mathbb{R}^3 \rightarrow \mathbb{R}$ given by 
$$
\psi^c(x) = \sum_{l=1}^3 \Bigl | \sum_{\alpha: | \alpha | \leq p} c_l(\alpha) x_1^{\alpha_1} x_2^{\alpha_2} x_3^{\alpha_3}\Bigr | ^{1+\eta} .
$$   
 Then
$$
\inf_{c \in \mathbb{R}^{3r}: \| c \|_1 = 1} \int_{\Lambda} \psi^c(x)  d\mu(x) \geq D > 0.
$$
\end{lemma}
\begin{proof}
Since $\Lambda$ has Hausdorff dimension strictly greater than 2, it is not contained in the zero set of any algebraic polynomial.  
Thus, for any fixed nonzero $c \in \mathbb{R}^{3r}$,
$$
G(c) :=  \int_{\Lambda} \sum_{j=1}^m \bigl | (A_x c)_j \bigr |^{1+\eta} d\mu(x) > 0.
$$
Since the set $c \in \mathbb{R}^{3r}: \| c \|_1 =1$ is compact and non-empty, we may apply the extreme value theorem: any continuous real-valued function over the space is bounded below and attains its infimum.  In particular, this implies
$$
\inf_{c \in \mathbb{R}^{3r}: \| c \|_1 = 1}  G(c) > 0.
$$
\end{proof}

\begin{theorem}[Main theorem]
\label{thm_main}
Fix $\eta > 0, \varepsilon > 0$, and maximal degree $p$.  Let $X^t=x(t) =(x_1(t),x_2(t),x_3(t))$ be the flow generated by a vector field $G \in {\mathcal U}$ whose governing equation $f: \mathbb{R}^3 \rightarrow \mathbb{R}^3$ in $\dot{x}(t) = f(x(t))$ is a multivariate algebraic polynomial of degree at most $p$,  and suppose we observe corrupted measurements of the time-1 map
$$U^t = X^t + \Theta_t, \quad V^t = \dot{X}^t + \Theta_t', \quad \quad t = 0, 1,2, \ldots, m $$
where $(\Theta, \Theta') \in \mathbb{R}^{6m}$ is sparse such that $\| (\Theta, \Theta') \|_0 \leq 2s$, and the sparse level $s$ is greater or equal than ${p+d \choose d}$, the maximal number of monomials of degree at most $p$. Also, assume that the underlying attractor for the flow has Hausdorff dimension greater than two.

 Then there are constants $C, C'$ depending only on  $\Lambda$, $p$, $B_{\Theta}$, $\eta$, and $\varepsilon$ such that if
$$
m \geq C, \quad s \leq C' m^{1/(1+\eta)},
$$ 
then the following holds with probability exceeding $1 - \varepsilon - e^{-d^3 \log(3m)}$ with respect to the initial condition $X^0 = x \sim d\mu$:
The polynomial coefficients of $f$, as well as the outlier vectors $(\Theta, \Theta')$, can be exactly recovered from the unique solution to the partial $\ell_1$-minimization problem
$$
\min\limits_{c, e} \hspace{1mm} \| e \|_1 \hspace{1mm} \emph{subject to} \hspace{1mm} v - A c = e.
$$
\end{theorem}

We turn to the proof of Theorem \ref{thm_main} shortly.  First, we provide a corollary of the theorem in case we observe measurements of the time-$\Delta$ map, for $\Delta < 1$.

\begin{corollary}
\label{cor_main}
Fix $\eta > 0, \varepsilon > 0$, $L \in \mathbb{N}$ such that $\Delta = \frac{1}{L}$, and degree $p$.  Let $X^t=x(t) =(x_1(t),x_2(t),x_3(t))$ be the flow generated by a vector field $G \in {\mathcal U}$ whose governing equation $f: \mathbb{R}^3 \rightarrow \mathbb{R}^3$ in $\dot{x}(t) = f(x(t))$ is a multivariate algebraic polynomial of degree at most $p$, and suppose we observe corrupted measurements of the time-$\Delta$ map
$$U^{j} = X^{\Delta j} + \Theta_j, \quad V^{j} = \dot{X}^{\Delta j} + \Theta_j', \quad \quad j = 0, 1,2, \dots, mL$$
where $(\Theta, \Theta') \in \mathbb{R}^{6m}$ is sparse and $\| (\Theta, \Theta') \|_0 \leq 2sL$. 

 Then for the same constants $C, C'$ as in Theorem \ref{thm_main}, once
$$
m \geq C, \quad s \leq C' m^{1/(1+\eta)},
$$ 
then the following holds with probability exceeding $1 - \varepsilon - e^{-d^3 \log(3m)}$ with respect to $x \sim d\mu$:
The polynomial coefficients of $f$, as well as the outlier vectors $(\Theta, \Theta')$, can be exactly recovered as the unique solution to the partial $\ell_1$-minimization problem
$$
\min\limits_{c, e} \hspace{1mm} \| e \|_1 \hspace{1mm} \emph{subject to} \hspace{1mm} v-Ac = e.
$$
\end{corollary}

\begin{proof}[Proof of Corollary \ref{cor_main}]
Consider the subsequences 
$$U_{k}^j = U_{k}^{Lj+k}, V_{k}^j = V^{Lj+k}, \quad k = 0,1, \dots, L-1, \quad  j = 0, 1, 2, \dots, m.$$
Each of these $L$ subsequences represents a corrupted measurement vector for a time-1 map of the flow $X^t$.  By the pigeonhole principle, one of these subsequences is sparsely corrupted, having associated corrupted vector of sparsity level $\| (\Theta_k, \Theta'_k) \|_0 \leq s$.  Thus we may apply Theorem \ref{thm_main}, using Corollary \ref{cor_cs} in place of Corollary \ref{prop_cs}. 
\end{proof}

\bigskip
\bigskip
 
 \begin{proof}[Proof of Theorem \ref{thm_main}]
 We break the proof into several parts. 

\bigskip
\bigskip

 \noindent First, consider a fixed polynomial coefficient vector $c = (c_1, c_2, c_3) \in \mathbb{R}^{3r}$ of unit norm $\| c \|_1 = 1$, and the corresponding vector 
 $v^c(x) = (v^c_1(x_1, x_2, x_3)), v^c_2(x_1, x_2, x_3), v^c_3(x_1, x_2, x_3))$ whose components are given by
 $$
 v^c_l(x_1, x_2, x_3) = \sum_{\alpha: | \alpha | \leq p} c_l(\alpha) x_1^{\alpha_1} x_2^{\alpha_2} x_3^{\alpha_3}.
 $$
  Consider the corresponding $C^{1+\eta}$ observable $\psi = \psi^c: \mathbb{R}^3 \rightarrow \mathbb{R}$ given by
\begin{equation}
 \label{observable}
 \psi^c(x) = | v^c_1(x) |^{1+\eta} + | v^c_2(x) |^{1+\eta} + | v^c_3(x) |^{1+\eta}.
\end{equation}
Combining the chain rule, Holder's inequality, and that 
 $\left| |x|^{\eta} - |y|^{\eta} \right| \leq | |x| - |y| |^{\eta}$ by concavity, we bound the $C^{1+\eta}$ norm of $\psi$ by
 \begin{align}
 \| \psi^c \|_{C^{1+\eta}} &\leq \max_{|\beta| \leq 1} \sup_{x \in \Lambda} | D^{\beta} \psi^c(x) | + \max_{|\beta| \leq 1} \sup_{x \in \Lambda} | D^{\beta} \psi^c(x) |  \nonumber \\
 &\leq  2\| c \|_1 (1+\eta) p \left( (B_{\Lambda})^{2p-1} \right)^{\eta} \nonumber \\
 &=   2(1+\eta) p  \left( (B_{\Lambda})^{2p-1} \right)^{\eta}  \nonumber \\
 & =: C_{\eta, \Lambda, p}. \nonumber 
 \end{align}
 In particular, $ \| \psi^c \|_{C^{1+\eta}}$ is bounded above by a constant which is independent of the number of measurements $m$.  Indeed, this bound holds uniformly over all observables corresponding to $c: \| c \|_1 = 1$:
 \begin{equation}
 \label{observe:uniform}
 \max_{c: \| c \|_1 = 1} \| \psi^c \|_{C^{1+\eta}} \leq C_{\eta, \Lambda, p}.
 \end{equation}
 We also have a uniform lower bound on a related quantity by Lemma \ref{lemma_lower}: let $D = D_{\eta, \Lambda, p} > 0$ be the lower bound in Lemma \ref{lemma_lower}. Then
 \begin{equation}
 \label{lower:uniform}
  \min_{c: \| c \|_1 = 1} \int_{\Lambda} \psi^c(x)  d\mu(x) \geq D_{\eta, \Lambda, p}.
 \end{equation}
 
 \bigskip
 \bigskip
 
 \noindent Recall that $m$ is our number of meausurements.  In order to apply a variant of Corollary \ref{CLT_main2} uniformly over all observables $\{\psi^c: \| c \| = 1\}$, we first discretize the set $\{ c: \| c \|_1 = 1\}$ using covering lemmas, and then apply a large deviations result. 
By a well-known result in the literature on covering numbers (see, e.g.,\cite{foucart2013mathematical}[ Appendix C.2]),  there exists a finite set of points ${\mathcal Q}$ in $\{ c: \| c \|_1 = 1\}$ such that 
$$
\max\limits_{c: \| c \|_1 = 1} \min\limits_{q \in {\mathcal Q}} \| c - q \|_1 \leq 1/m,
$$
of cardinality
$$
|{\mathcal Q} | \leq (3m)^{3r}.
$$
We now apply Corollary \ref{CLT_main2} of Theorem \ref{CLT_main} uniformly over the observables $\psi^q: q \in {\mathcal Q} $. Draw an initial condition $X^0 = x$ from the measure $d\mu$.  The following holds with probability $1-\varepsilon$ with respect to the draw of $x$:
 \begin{eqnarray}
&& \min_{q \in {\mathcal Q}} \left| \sum_{j=0}^{m-1} \psi^q( X^j) \right| \nonumber \\
\quad &\geq& \min_{q \in {\mathcal Q}}  \left| m \int_{\Lambda} \psi^q(x) d\mu(x) \right| - \max_{q \in {\mathcal Q}}  \left| \sqrt{C_1} C_{\eta, \Lambda, p} \sqrt{m} Z^q \right| -  \max_{q \in {\mathcal Q}}  \kappa_{\varepsilon, p, \eta} m^{1/4} (\log(m))^{1/2} (\log \log m)^{1/4} ) \nonumber\\
\quad &\geq& m D_{\eta, \Lambda, p}  - \sqrt{m} \sqrt{C_1} C_{\eta, \Lambda, p}  \max_{q \in {\mathcal Q}} | Z^q |  - \kappa_{\varepsilon, p, \eta} m^{1/4} (\log(m))^{1/2} (\log \log m)^{1/4} ),\nonumber
\end{eqnarray}
where $Z^q$ denotes a standard normal random variable. 
\bigskip
\noindent We now bound $ \max\limits_{q \in {\mathcal Q}} | Z^q |.$ First, recall the Chernoff bound for a standard Gaussian random variable $Z$:
\[
P(|Z| \geq t) \leq 2e^{-t^2/2}, \quad \forall t \geq 0.
\]
Recalling that $| {\mathcal Q} |  \leq (3n)^{3r},$ the union bound then gives
$$
P(\exists q \in {\mathcal Q}: |Z^q| \geq t) \leq 2(3m)^{3r} e^{-t^2/2} \leq  2e^{3r \log(3m) - t^2/2}, \quad \forall t \geq 0.
$$
In particular,
$$
P(\forall q \in {\mathcal Q}: \hspace{1mm} | Z^q | \leq 2\sqrt{3r \log(3m)}  ) \geq 1- 2 e^{-3r \log(3m)}.
$$
\bigskip
\noindent 
All together, we find with probability exceeding  $1 - \varepsilon  - 2 e^{-3r \log(3m)}$ with respect to $X^0 \sim d\mu$, 
 \begin{eqnarray}
 \min_{q \in {\mathcal Q}}  \left| \sum_{j=0}^{m-1} \psi^q( X^j) \right| &\geq&  m D - \sqrt{m \log(m) } C  - \kappa m^{1/4} (\log(m))^{1/2} (\log \log m)^{1/4}, \nonumber
 \end{eqnarray}
 where the constants $D, C$, and $\kappa$ depend only on $\varepsilon, \eta, \Lambda,$ and $p$. Thus, for a sufficiently large constant $C' = C'(\varepsilon, d, \eta)$ and sufficiently small constant $C''= C''(\varepsilon, d, \eta)$,  and using the inequality $\| v \|_1 \geq \| v \|_{1+\eta}$,
the following uniform lower bound holds with probability exceeding  $1 - \varepsilon  - 2 e^{-3r \log(3m)}$ with respect to $X^0 \sim d\mu$:
if the number of measurements satisfies
\begin{equation}
\label{result_net}
m \geq C',
\end{equation}
then
 \begin{eqnarray}
 \label{result_net2}
 \min_{q \in {\mathcal Q}}  \left( \sum_{j=0}^{m-1} | v^q_1(x) | + | v^q_2(x) | + | v^q_3(x) | \right) \geq (C'' m)^{1/(1+\eta)}. 
\end{eqnarray}

\bigskip
\bigskip

\noindent Recall now by \eqref{range} that 
$$
 \sum_{j=1}^{3m} |  (A q)_j | = \sum_{j=0}^{m-1} \psi^q( X^j) + \sum_{j=1}^m \sum_{l=1}^3 |R_l(j)|.
 $$
Assuming \eqref{result_net} and with the same probability, and using the bound \eqref{bound_residual} on $\| R_l \|_{\infty}$ and that $\| R_l \|_{0} \leq 2s$,
 \begin{eqnarray}
 \label{result_net3}
 \min_{q \in {\mathcal Q}}  \sum_{j=1}^{3m} |  (A q)_j | &\geq&  (C'' m)^{1/(1+\eta)} - \sum_{j=1}^m \sum_{l=1}^3 |R_l(j)| \nonumber \\
 &\geq& (C'' m)^{1/(1+\eta)}  - 6s (B_{\Theta}+B_{\Lambda})^d.   \nonumber 
\end{eqnarray}

\bigskip
\bigskip

\noindent We now use a continuity argument to pass this lower bound from the discrete net ${\mathcal Q}$ to the entire sphere $\{ c: \| c \|_1 = 1 \}$. Fix $c \in \{ c: \| c \|_1 = 1 \}$ and let $q \in {\mathcal Q}$ be such that $\| c - q \|_1 \leq 1/m$, which exists by construction of ${\mathcal Q}$.   By Holder's inequality,
\begin{align}
\| A (c - q) \|_1 &\leq 3m B_{\Lambda}^d \| c - q \|_1 + s (B_{\Lambda}+B_{\Theta})^d \| c - q \|_1 \nonumber \\
 &\leq  3B_{\Lambda}^d + 3\frac{s}{m} (B_{\Lambda}+B_{\Theta})^d. 
\end{align}
Thus, 
\begin{eqnarray}
 \min_{ c: \| c \|_1 = 1 } \| A c \|_1 &\geq& (C'' m)^{1/(1+\eta)}  - C'''s (B_{\Lambda}+ B_{\Theta})^d.  \nonumber
\end{eqnarray}

At the same time, for a subset $S \subset \{1,2,\ldots,m\}$ of size $s$, we have the immediate and uniform upper bound
$$
 \sum_{j \in S} | (A c)_j | \leq s B_{\Lambda}^d  \| c \|_1. 
$$

Thus, there is a constant $C''''$ depending on only $\varepsilon, p, \eta$, and $B_{\Theta})$ such that if $s \leq  C'''' m^{1/(1+\eta)}$, then uniformly over $c: \| c \|_1 = 1$, and uniformly over subsets $S \subset [3m]$ of size $s$,
\begin{equation}
\label{eq:expression}
\| (Ac)_S \|_1 < \frac{1}{2} \| A c \|_1.
\end{equation}
and in particular, $ \| Ac  \|_1 > 0$ for all $c: \| c \|_1 = 1$, implying that the columns of $A$ are linearly independent.  

\bigskip
\bigskip

\noindent The theorem follows for $c \neq 0$ of arbitrary $\ell_1$ norm by normalizing both sides of the expression \eqref{eq:expression}.  The theorem follows by application of Corollary \ref{prop_cs}.
\end{proof}

As a consequence of our proof, we also provide theoretical guarantees for the algorithm of \cite{brunton2015discovering} in the noiseless case.
\begin{corollary}
\label{cor:fullrank}
Under the same conditions as in Theorem \ref{thm_main}, the matrix $\Phi$ constructed from uncorrupted measurements of the time-$\Delta$ map $X_j = X^{\Delta_j}$ is full rank, provided that $m\geq C'$.
\end{corollary}

\section{Numerical Method}
In this section, we explain how to solve our proposed model \eqref{model:main} numerically by using augmented Lagrangian/Bregman distance and alternating minimization method. Recall our proposed model:
\begin{equation*}
\begin{aligned}
\min\limits_{(\mathcal C,\mathcal E)} \, \|{\mathcal E}\|_{2,1} =&\min\limits_{({\mathcal C},{\mathcal E})} \, \sum\limits_{j=1}^m\, \|{\mathcal E}(j,:)\|_2,\\
&\text{subject to}\quad \Phi(X^\circ){\mathcal C} +{\mathcal E}  =\dot{X^\circ} \quad \text{and  \,}{\mathcal C} \ \text{is sparse.}
\end{aligned}
\end{equation*}
The corresponding augmented Lagrangian is of the form
\begin{equation}
\begin{aligned}
(\mathcal C^{k+1},\mathcal E^{k+1}) &= \min\limits_{({\mathcal C},{\mathcal E})} \,  \sum\limits_{j=1}^m\, \|\mathcal E(j,:)\|_2 + \dfrac{\mu}{2} \| \Phi(X^\circ)\mathcal C + \mathcal E - \dot{X^\circ} + b^k\|_F^2,\\
b^{k+1} &= b^k +\Phi(X^\circ){\mathcal C}^{k+1} + {\mathcal E}^{k+1} -\dot{X^\circ}.
\end{aligned}
\label{eqn:augmented}
\end{equation}\textbf{}
Now we can apply the alternating minimization method to solve problem  \eqref{eqn:augmented}.
\begin{itemize}
	\item The $\mathcal C$-subproblem:
	$${\mathcal C}^{k+1} = \min\limits_{\mathcal C}  \| \Phi(X^\circ)\mathcal C + \mathcal E^{k} - \dot{X^\circ} + b^k\|_F^2,\quad\text{s.t. $\mathcal C$ is sparse.}$$

\item The $\mathcal E$-subproblem: 
	$$\mathcal E^{k+1} = \min\limits_{\mathcal E} \sum\limits_{j=1}^m\, \|\mathcal E(j,:)\|_2 + \dfrac{\mu}{2} \| \Phi(X^\circ)\mathcal C^{k+1} + \mathcal E - \dot{X^\circ}+b^k\|_F^2.$$
\end{itemize}

Notice that the $\mathcal C$-subproblem is a least-squares problem for an over-determined system. The sparsity of $\mathcal C$ is a property of the system itself, therefore we enforce the sparsity of $\mathcal C$ by applying the hard-thresholding operator after obtaining the least-squares solution:
$$\mathcal C^{k+1}=S_h\left((\Phi(X^\circ))^{-1} (\dot{X^\circ}-\mathcal E^k-b^k),\lambda\right),$$
where 
$$S_h(u,\gamma) := u\cdot I_{\abs{u}\geq\gamma} = \begin{cases} &u \quad\text{if} \quad \abs{u}\geq\gamma\\
& 0\quad\text{otherwise}.
\end{cases}
$$
It is also discussed in \cite{brunton2015discovering} that the mentioned method is robust to noise in recovering the true coefficients $\mathcal{C}$.

The solution for $\mathcal E$ is given explicitly
$$\mathcal E^{k+1}= S_{2} \left(\dot{X^\circ} -b^k-\Phi(X^\circ)\mathcal C^{k+1},\, \mu \right),$$
where
$$S_2(u_j,\gamma) = \max \left(1 -\dfrac{1}{\gamma \|u_j\|_2},\,0\right) u_j,$$
for each row $u_j$ of $u$.

Below is the summary of the algorithm for problem \eqref{eqn:augmented}.

\bigskip

\noindent\fbox{%
\begin{minipage}{\dimexpr\linewidth-2\fboxsep-2\fboxrule\relax}
\begin{algorithmic}
\State \underline{\textbf{Algorithm}}
\State Given: $\mathcal E^0,b^0, tol $ and parameters $\lambda, \mu$.
\medskip
\While{$||\mathcal E^{k}-\mathcal E^{k-1}||_{\infty}> tol$}
\medskip
\State $\mathcal C^{k+1}=S_h\left((\Phi(X^\circ))^{-1} (\dot{X^\circ}-\mathcal E^k-b^k),\lambda\right)$
\medskip
\State $\mathcal E^{k+1}= S_{2} \left(\dot{X^\circ} -b^k-\Phi(X^\circ)\mathcal{C}^{k+1},\, \mu \right)$
\medskip
\State $b^{k+1}=b^k+\Phi(X^\circ)\mathcal{C}^{k+1} + \mathcal E^{k+1} -\dot{X^\circ}$
\medskip
 \EndWhile
\end{algorithmic}
\end{minipage}%
}

 \bigskip
 
\section{Numerical Results}
In this section, we apply the method from previous section to various chaotic systems including the well-known Lorenz system and R{\" o}ssler system. Moreover, our proposed reconstruction method also works numerically for systems exhibiting so-called \emph{hyperchaos} whose dimension is greater than three, which suggests that our reconstruction guarantee may extend to higher dimensional systems.  We reiterate that chaotic systems are not only well-suited for theoretical reconstruction guarantees, but also regimes where it is of upmost importance to recover the governing equations with high precision, in light of the property of sensitivity to initial conditions.  Therefore, we define the following relative error formula for the coefficients
\begin{equation*}
(\text{coefficient})\, \text{error}:= \max\left(\max\limits_{{\mathcal C}_{\text{true}}(i) \not = 0} \left |\dfrac{{\mathcal C}_{\text{recovered}}(i) - {\mathcal C}_{\text{true}}(i)}{{\mathcal C}_{\text{true}}(i)}\right |, \max\limits_{\substack{{\mathcal C}_{\text{recovered}}(i) \not =0,\\ {\mathcal C}_{\text{true}} (i)=0}}{| {\mathcal C}_{\text{recovered}}(i)|} \right).
\end{equation*}
In words, the coefficient error measures the maximal relative recovery error over the different polynomial coefficients. Throughout this section, error stands for coefficient error, unless otherwise stated.

In practice, we will not observe the derivative information $\dot{X}^o$.  Instead, we can approximate the rate of change in the system from the state space measurements using first-order, second-order or higher-order approximations.  Explicitly, given $x(t)\in\mathbb{R}^n$, the rate of change in $x$ can be approximated as
\begin{equation*}
\begin{aligned}
\dot{x}(t) &= \dfrac{x(t+dt) - x(t)}{dt} + \mathcal{O}(dt),\quad\text{(first-order approximation)}\\
\dot{x}(t) &= \dfrac{x(t+dt) - x(t-dt)}{2\,dt} + \mathcal{O}(dt^2), \quad\text{(second-order approximation)}
\end{aligned}
\end{equation*}
To get better accuracy, we use the second-order approximation of rate of change for our optimization model.

\bigskip

We first show some numerical results for the Lorenz system
\begin{equation}
\begin{cases}
\frac{dx_1}{dt} &= 10(x_2-x_1)\\
\frac{dx_2}{dt} & = x_1(28 -x_3) -x_2\\
\frac{dx_3}{dt} & = x_1x_2 - \frac{8}{3} x_3,
\end{cases}
\label{eqn:lorenz}
\end{equation}
with different percentages of corruption. To simulate the measurement data, we first solve the Lorenz system \eqref{eqn:lorenz} using the fourth-order Runge-Kutta method with $dt = 0.0005$. Then we randomly assign locations where the data is corrupted over intervals of time, and add Gaussian noise with standard deviation $\sigma$ to the data at those corrupted intervals. The bandwidth of corruption ranges from 5 to 50. From the simulated data, we build the matrix $X^\circ$, compute the time derivative $\dot{X^\circ}$ using the second-order approximation of derivatives, and build the dictionary $\Phi(X^\circ)$. We verify our algorithm for the integration (from $t=0$ to $t=20, dt = 0.0005$) of 40000 measurements with different percentages of corruption. The results are shown in Figure \ref{fig:lorenz20} and Figure \ref{fig:lorenz50}. In all cases, our algorithm can detect exactly the locations of the outliers and recover the coefficients in the polynomial equations with very high accuracy.  Notice that our model can tolerate a high percentage of corruption as long as the size of the data is sufficiently large.  
  \begin{figure}[!ht]
\begin{minipage}[b]{0.48\linewidth}
\centering
\includegraphics[width  =\linewidth]{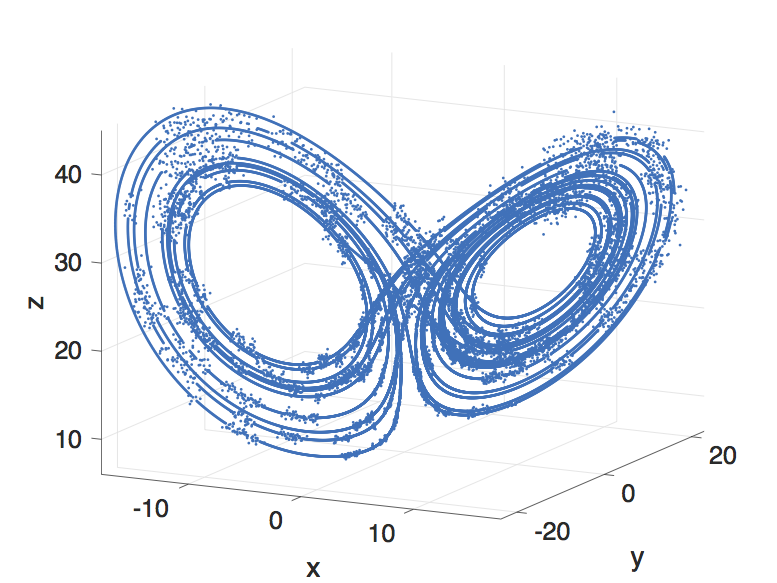}
\vspace{-3. cm}
 \end{minipage}
 \begin{minipage}[b]{0.48\linewidth}
 \centering
  \begin{tabular}{ |c || c | c | c | }
    \hline
     & $ \dot{x_1}$  & $\dot{x_2}$ & $\dot{x_3}$ \\ \hline
    1 & 0 & 0  &  $0$ \\ \hline
    $x_1$ & -9.999947 & 27.9995 & 0 \\ \hline
    $x_2$ & 9.999949 & -0.9999 & 0 \\ \hline
    $x_3$ & 0  & 0 & -2.666648 \\ \hline
    $x_1^2$ & 0 & 0 & 0 \\ \hline
    $x_1x_2$ & 0 & 0 & 0.999993 \\ \hline
    $x_1x_3$ & 0 & -0.999986 & 0 \\ \hline
    $x_2^2$ & 0 & 0 & 0 \\ \hline
    \vdots & \vdots &\vdots &\vdots \\ \hline
    $x_3^4$ & 0 & 0 &0 \\ \hline
  \end{tabular}
  \end{minipage}
\caption{Left: Lorenz system \eqref{eqn:lorenz}, with 19.19\% corrupted data, Tfinal = 20, dt = 0.0005, hard-thres = 0.1, row-thres = 0.0125, tol = 0.005. Right: the recovered coefficients. The model recovers the coefficients within 0.0096\% error and detect exactly the locations of the outliers after 22 iterations.}
\label{fig:lorenz20}
\end{figure}

\begin{figure}[!ht]
\centering
\includegraphics[width = 2.1 in ]{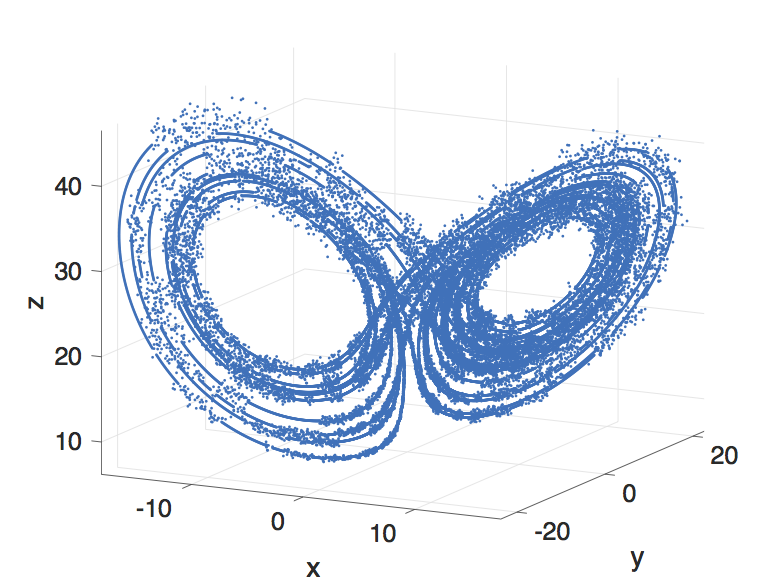}\
\includegraphics[width = 2.1 in]{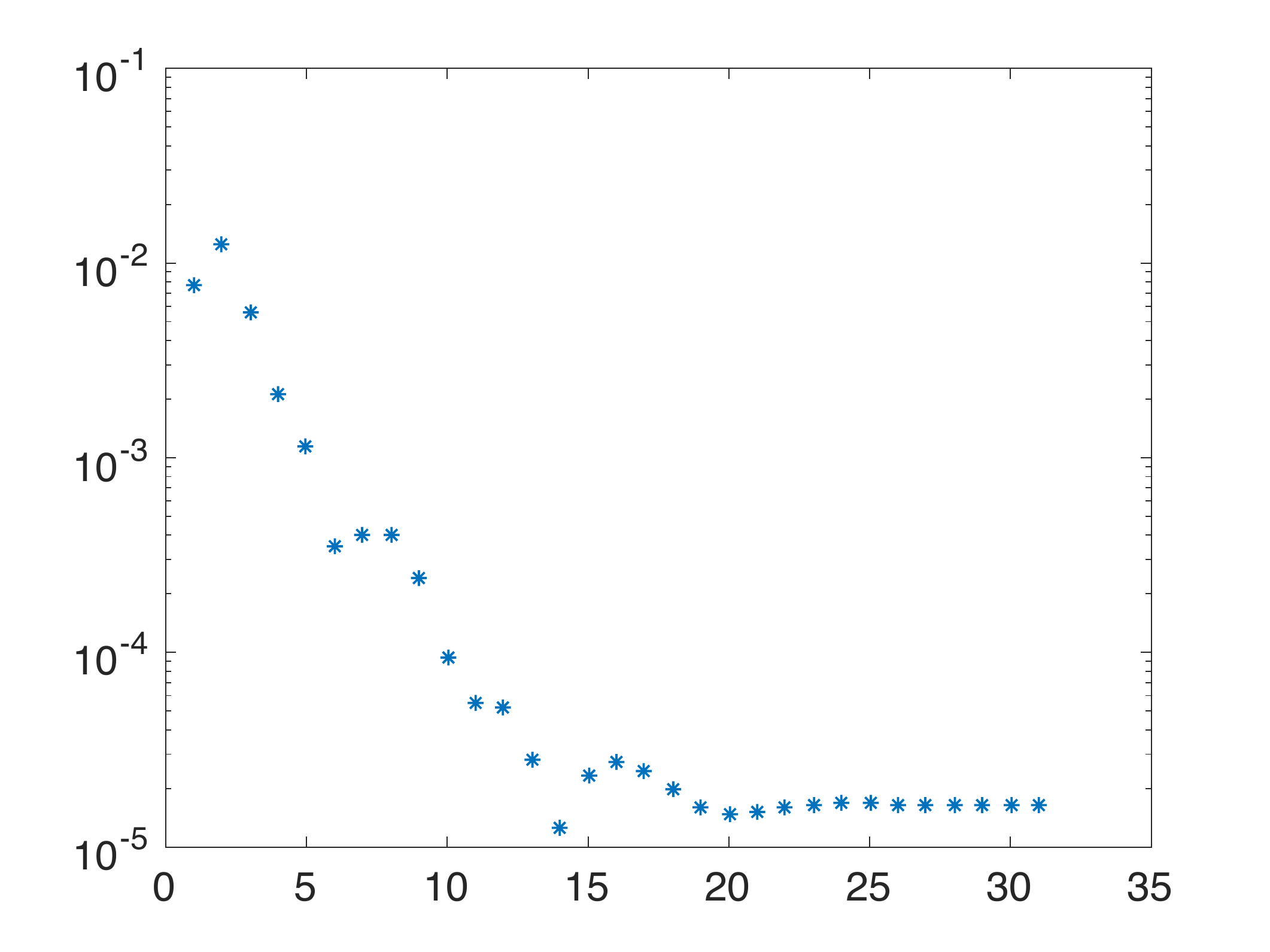}\
\includegraphics[width = 2.1 in]{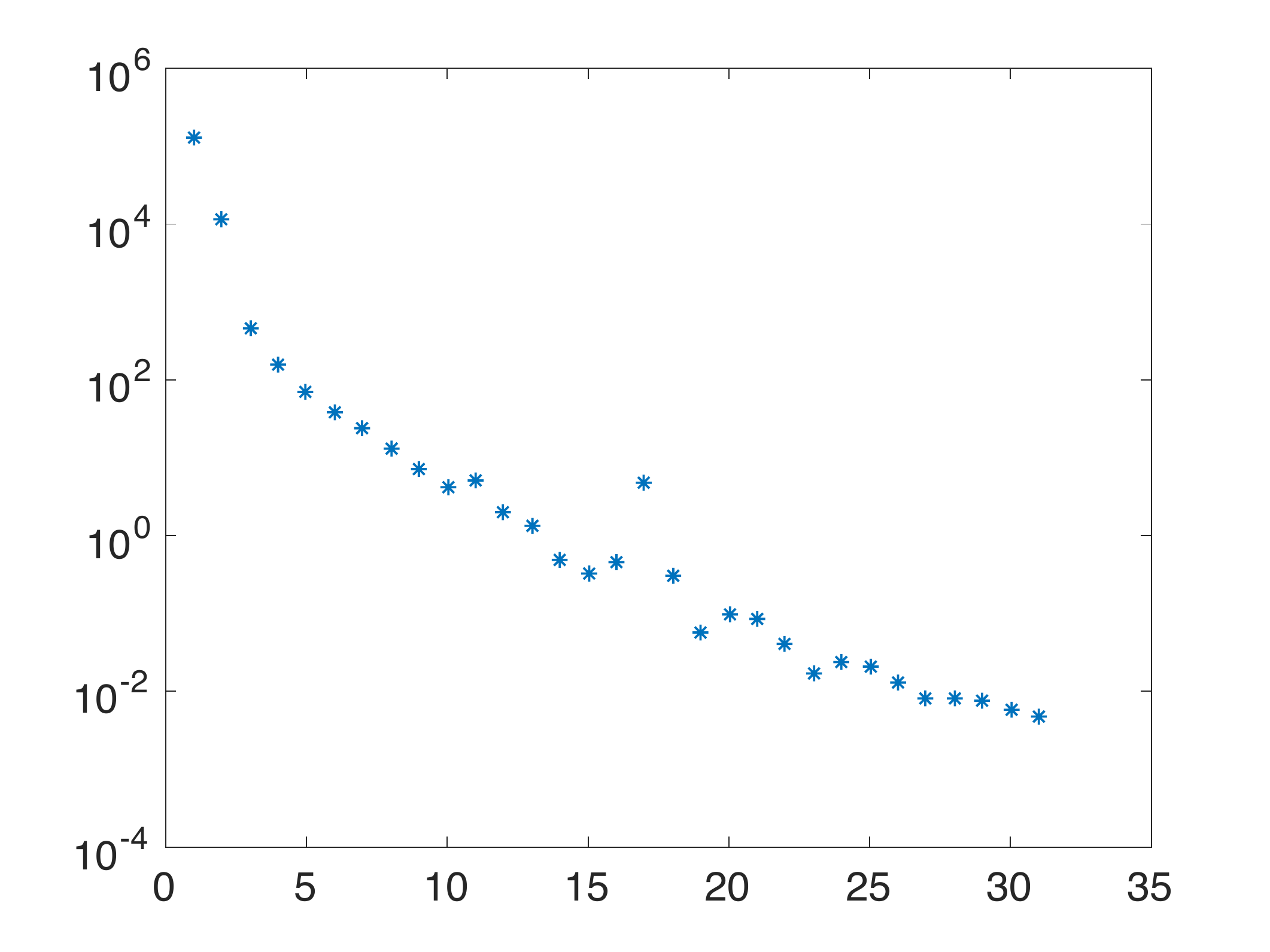}
\caption{Left: Lorenz system \eqref{eqn:lorenz}, with 49.53\% corrupted data, Tfinal = 20, dt = 0.0005, hard-thres = 0.1, row-thres = 0.0125, tol = 0.005. The error between the recovered coefficients iterates and the true ones (middle) and the error between two consecutive $\mathcal{E}$ (right) versus the number of iterations (in logarithmic scale for the vertical axis).  The model recovers the coefficients with 0.0096\% error and detect exactly the locations of the outliers after 31 iterations.}
\label{fig:lorenz50}
\end{figure}
We also examine our proposed scheme for the Lorenz system over a shorter interval of time.  In Figure \ref{fig:lorenz20_short}, 5000 measurements are given with 22.55\% and 71.89\% corrupted, corresponding to $\text{dt} = 0.0005$ and $\text{Tfinal} = 2.5$.  All other parameters remain the same as in the previous example. The model recovers the coefficients within 0.0317\% and 0.0477\% errors, respectively, and detects exactly the locations of outliers after 24 iterations. If the number of measurements is smaller, for example if $\text{Tfinal} \leq 2$, then the scheme does not work as well, illustrating the importance of observing the system over a critical minimal amount of time. 

\begin{figure}[!ht]
\centering
\includegraphics[width = 3. in ]{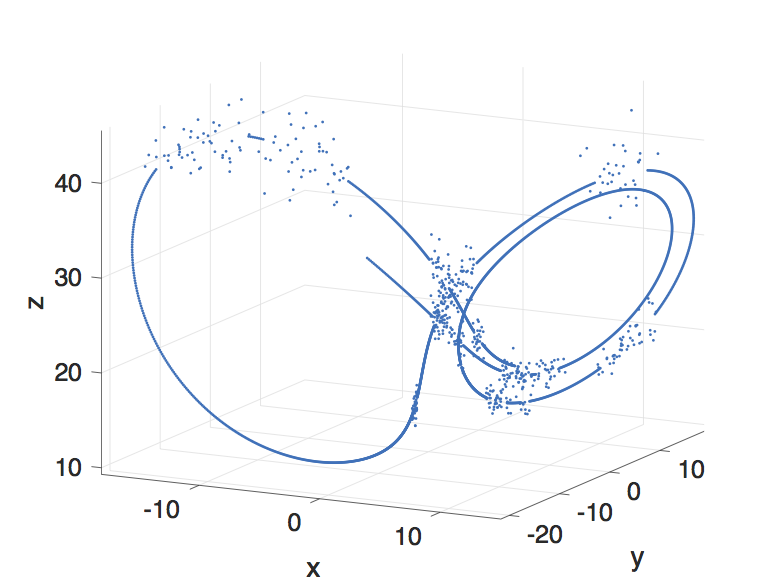}\quad
\includegraphics[width = 3. in]{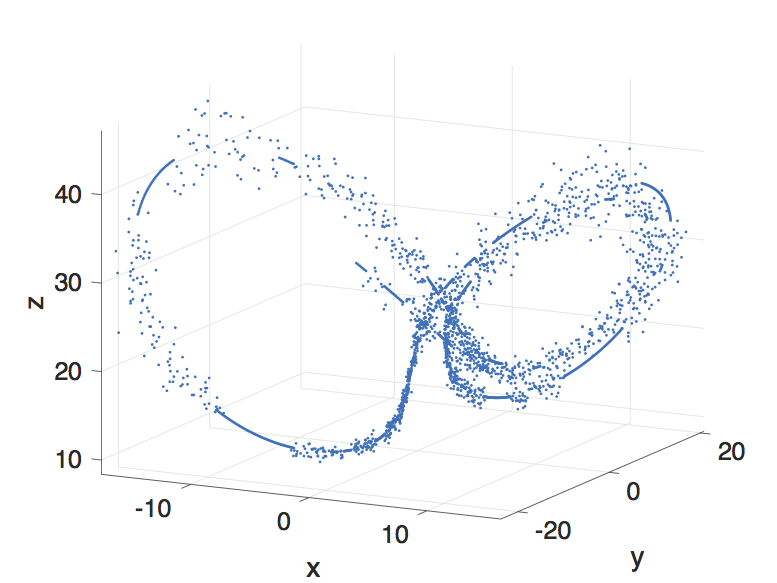}
\caption{Lorenz trajectories with 22.55\% corruption (left) and with 71.89\% corruption (right), associated with the system \eqref{eqn:lorenz}, Tfinal = 2.5, dt = 0.001, hard-thres = 0.1, row-thres = 0.0125, tol = 0.005. The model detects exactly the locations of the outliers in both data and recovers the coefficients with the errors 0.0317\% and 0.0477\%, respectively.}
\label{fig:lorenz20_short}
\end{figure}

Finally, we verify our proposed scheme for the Lorenz system over a much longer interval of time time, in order to gauge whether the butterfly effect (sensitivity to initial conditions) begins to interfere with the recovery conditions.  In Figure \ref{fig:lorenz100}, 200000 measurements are taken with $19.75\%$ corruption, corresponding to $\text{dt = .0005}$ and $\text{Tfinal = 100}$ and the other parameters as in previous examples. The model recovers the coefficients within $0.0097\%$ error and detect exactly the locations of the outliers after 29 iterations.
\begin{figure}
\centering
\includegraphics[width = 2.1 in ]{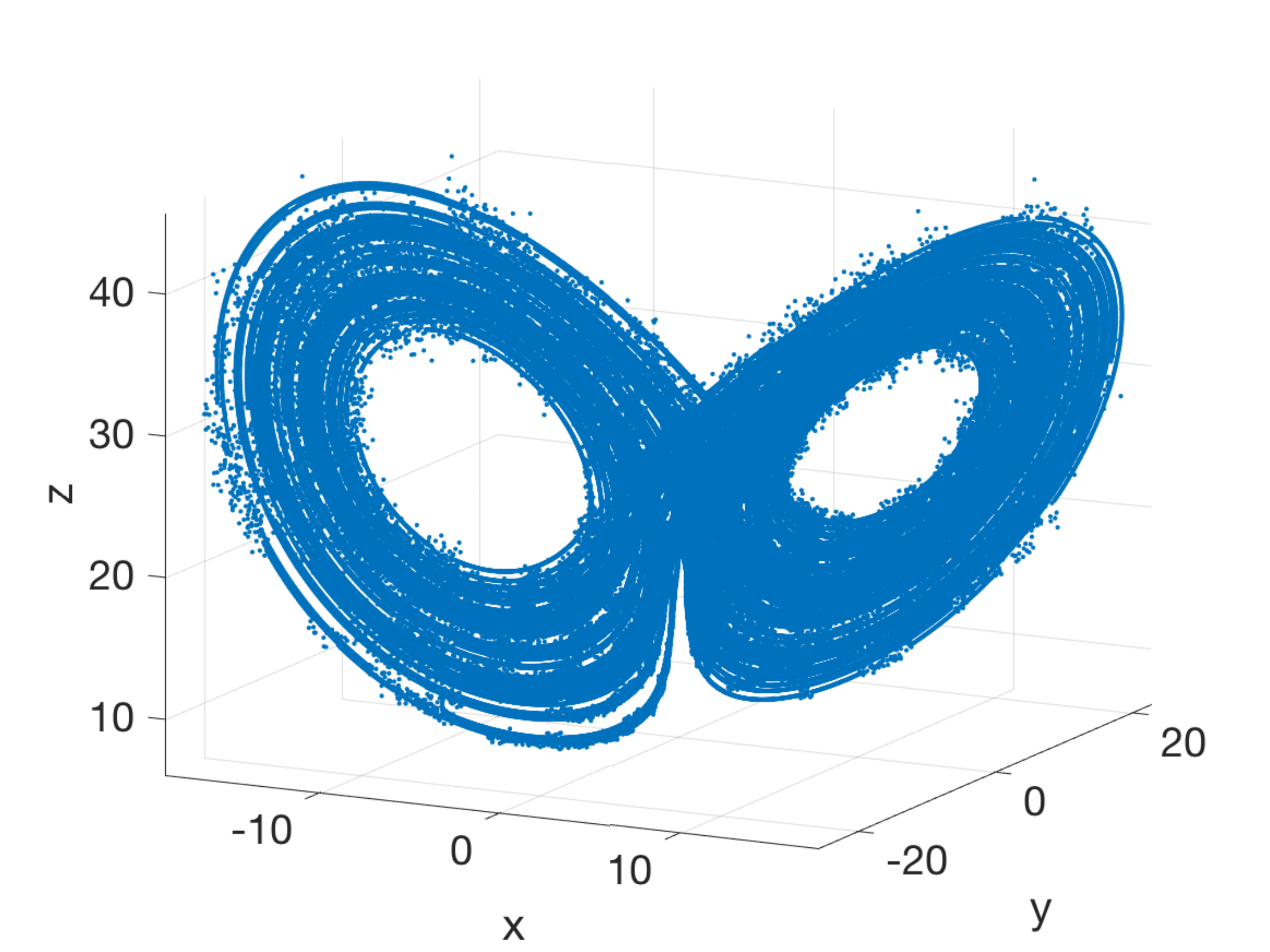}\
\includegraphics[width = 2.1 in]{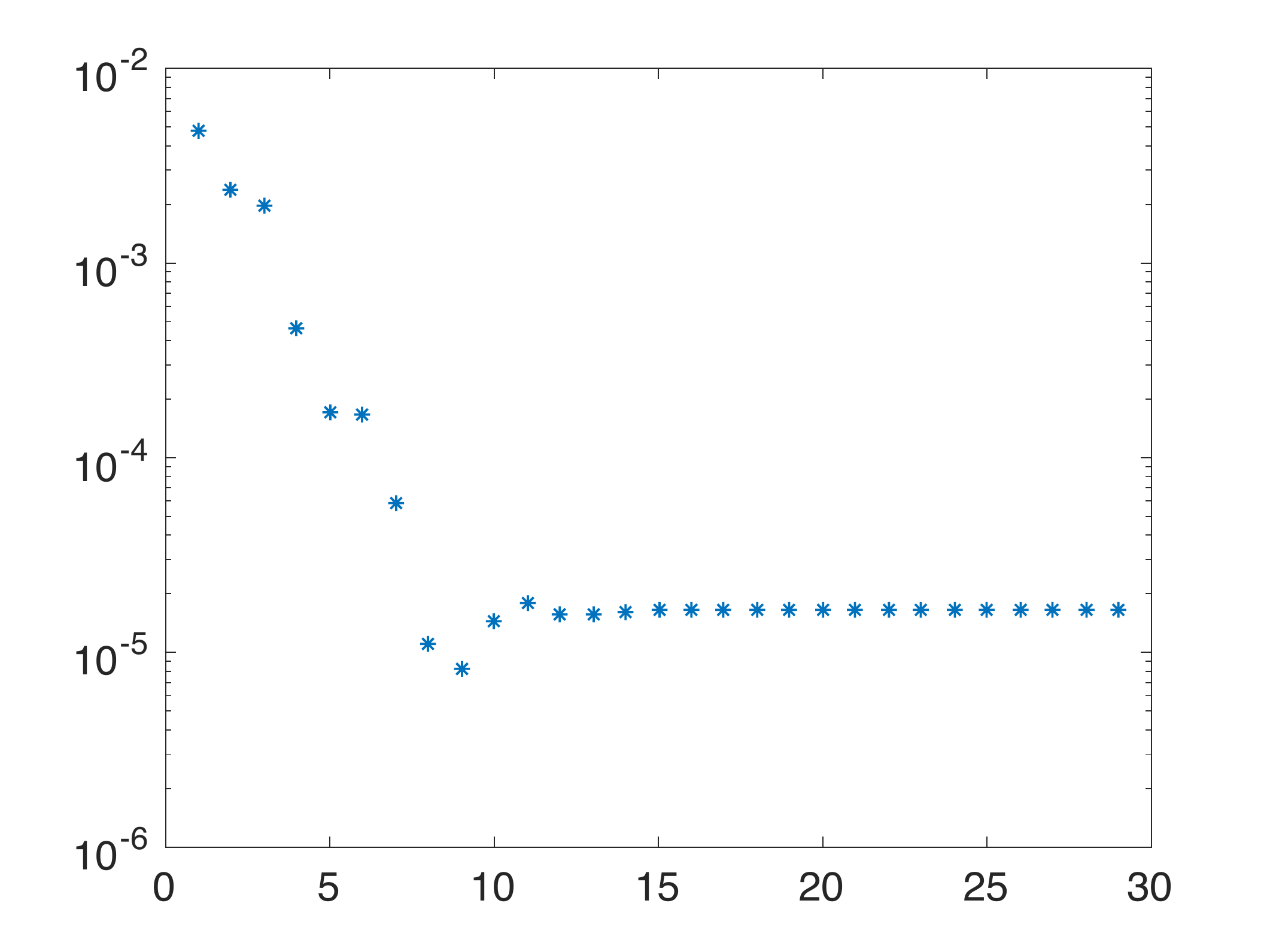}\
\includegraphics[width = 2.1 in]{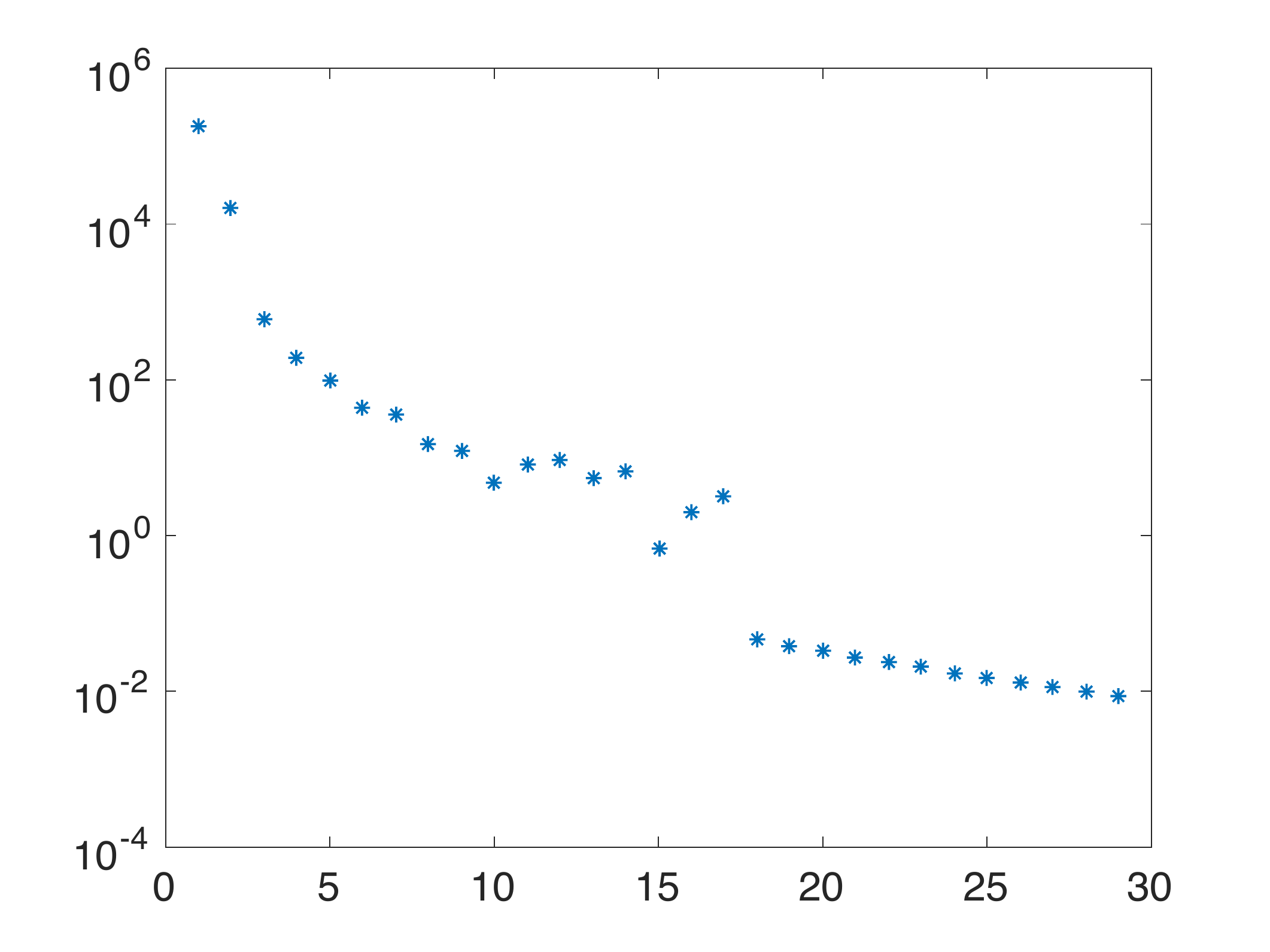}
\caption{Left: Lorenz system \eqref{eqn:lorenz}, with 19.75\% corrupted data, Tfinal = 100, dt = 0.0005, hard-thres = 0.1, row-thres = 0.0125, tol = 0.01. The error between the recovered coefficients iterates and the true ones (middle) and the error between two consecutive $\mathcal{E}$ (right) versus the number of iterations (in logarithmic scale for the vertical axis).  The model recovers the coefficients with 0.0097\% error and detect exactly the locations of the outliers after 29 iterations.}
\label{fig:lorenz100}
\end{figure}
In the next step, we show the robustness of our model to a small amount of additional additive noise.  After simulating the corrupted data from the Lorenz system, we add Gaussian noise to the entire data.  From the noisy data, we build the dictionary and approximate the time derivative. To test the robustness of our model, we simulate 100 sets of data and check how many times our model can detect exactly the locations of outliers. The summary is shown in Table 1. 

\bigskip
\begin{table}
\begin{tabular}{|c |c|c|}
\hline
 Standard Deviation of Noise & \# Times (out of 100) Outliers Detected Exactly & Coefficient Error (\%)   \\
 \hline
 0.4*dt & 89 &  $\min = 0.0009, \max =  0.0525$\\
 \hline
 0.6*dt & 87 & $\min = 0.0006, \max = 0.9395$\\
 \hline 
 0.8*dt & 65 &  $\min = 0.0012, \max = 1.57 $\\
 \hline
\end{tabular}
\caption{Different noise levels and the recovery results associated with the Lorenz system, Tfinal = 20, dt = 0.0005, hard-thres = 0.1, row-thres = 0.0125, tol = 0.005 and around $20\%$ corrupted.} 
 \end{table}
 Next, we examine the R{\" o}ssler system 
 \begin{equation}
\begin{cases}
\frac{dx_1}{dt} &= -x_2-x_3\\
\frac{dx_2}{dt} & = x_1+0.2\, x_2\\
\frac{dx_3}{dt} & = 0.2 - 5.7x_3 + x_1x_3,
\end{cases}
\label{eqn:rossler}
\end{equation}
with different percentages of corruption. We simulate the data in the same way as for the Lorenz system. Our algorithm can detect exactly the locations of outliers in all such cases, and recover the constant coefficient as well as other coefficients with very high accuracy. The numerical results are shown in Figure \ref{fig:Rossler11} to Figure \ref{fig:Rossler40}.

 \begin{figure}
 \begin{minipage}[b]{0.48\linewidth}
\centering
\includegraphics[width = \linewidth]{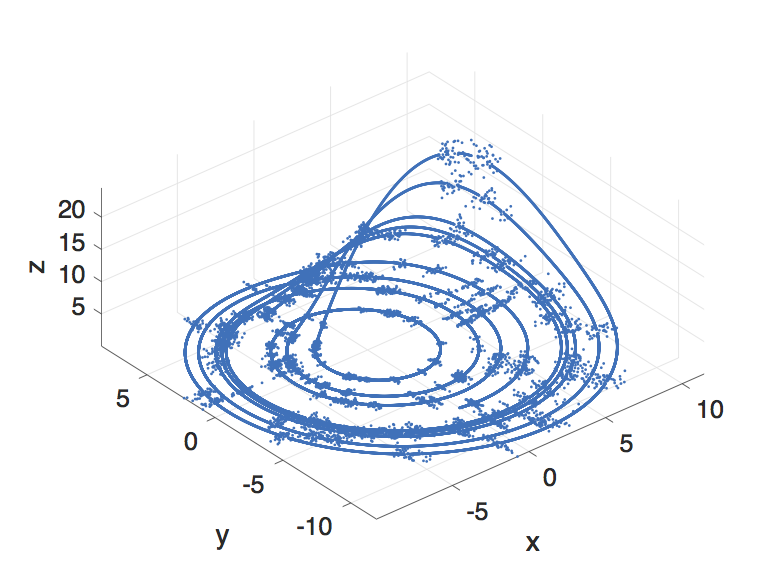}
\vspace{-3cm}
\end{minipage}
   \begin{minipage}[b]{0.48\linewidth}
   \centering
  \begin{tabular}{ |c || c | c | c | }
    \hline
     & $ \dot{x_1}$  & $\dot{x_2}$ & $\dot{x_3}$ \\ \hline
    1 & 0 & 0  &  $0.2012$ \\ \hline
    $x_1$ & 0 & 0.9994 & 0 \\ \hline
   $x_2$ & -1.0021 & 0.1948 & 0 \\ \hline
    $x_3$ & -1.0044  & 0 & -5.7009 \\ \hline
    $x_1^2$ & 0 & 0 & 0 \\ \hline
    $x_1x_2$ & 0 & 0 & 0.00 \\ \hline
    $x_1x_3$ & 0 & 0 & 1.0009 \\ \hline
    $x_2^2$ & 0 & 0 & 0 \\ \hline
    \vdots & \vdots &\vdots &\vdots \\ \hline
    $x_3^4$ & 0 & 0 &0 \\ \hline
     \end{tabular}
\end{minipage}   
 \caption{Left: R{\" o}ssler system \eqref{eqn:rossler} with 10.2\% corrupted data, Tfinal = 50, dt = 0.0005, hard-thres = 0.05, row-thres = 0.025, tol = $10^{-5}$. Right: the recovered coefficients. The model recovers the coefficients with 0.6\% error and detect exactly the locations of the outliers after 160 iterations.}
 \label{fig:Rossler11}
\end{figure}

\begin{figure}
\begin{minipage}[b]{0.48\linewidth}
\centering
\includegraphics[width = \linewidth]{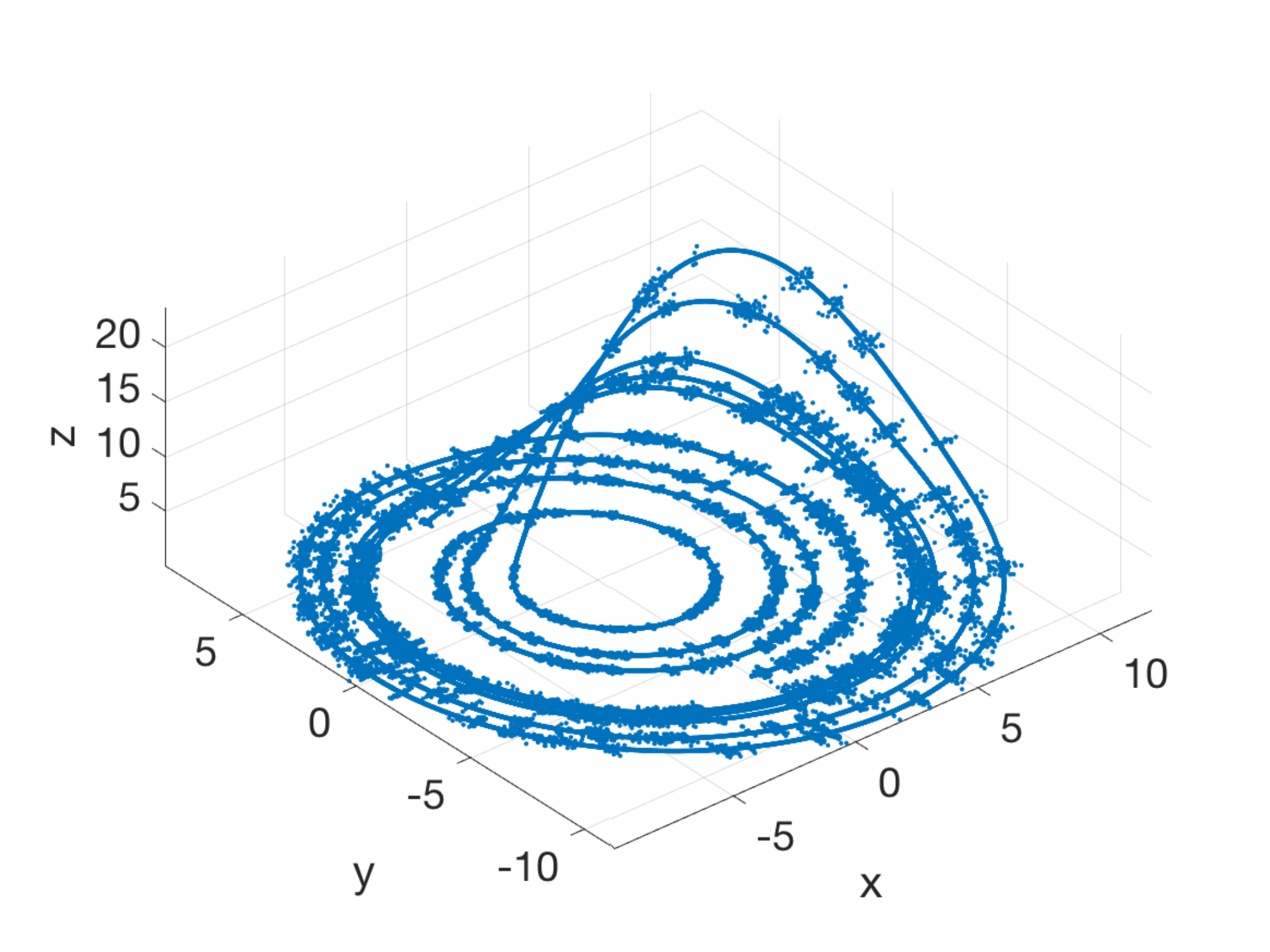}
\vspace{-3. cm}
\end{minipage}
  \begin{minipage}[b]{0.48\linewidth}
  \centering
  \begin{tabular}{ |c || c | c | c | }
    \hline
     & $ \dot{x_1}$  & $\dot{x_2}$ & $\dot{x_3}$ \\ \hline
    1 & 0 & 0  &  $0.2007$ \\ \hline
    $x_1$ & 0 & 1.0017 & 0 \\ \hline
    $x_2$ & -1.0019 & 0.1975 & 0 \\ \hline
    $x_3$ & -1.0049  & 0 & -5.7142 \\ \hline
    $x_1^2$ & 0 & 0 & 0 \\ \hline
    $x_1x_2$ & 0 & 0 & 0.00 \\ \hline
    $x_1x_3$ & 0 & 0 & 1.0025 \\ \hline
    $x_2^2$ & 0 & 0 & 0 \\ \hline
    \vdots & \vdots &\vdots &\vdots \\ \hline
    $x_3^4$ & 0 & 0 &0 \\ \hline
    \hline
  \end{tabular}
\end{minipage}
\caption{Left: R{\" o}ssler system \eqref{eqn:rossler} with 20\% corrupted data, Tfinal = 50, dt = 0.0005, hard-thres = 0.05, row-thres = 0.025. Right: the recovered coefficients. The model recovers the coefficients with 1.25\% error and detect exactly the locations of the outliers after 160 iterations.}
\label{fig:Rossler20}
\end{figure}

\begin{figure}
\begin{minipage}[b]{0.48\linewidth}
\centering
\includegraphics[width = 2.5 in]{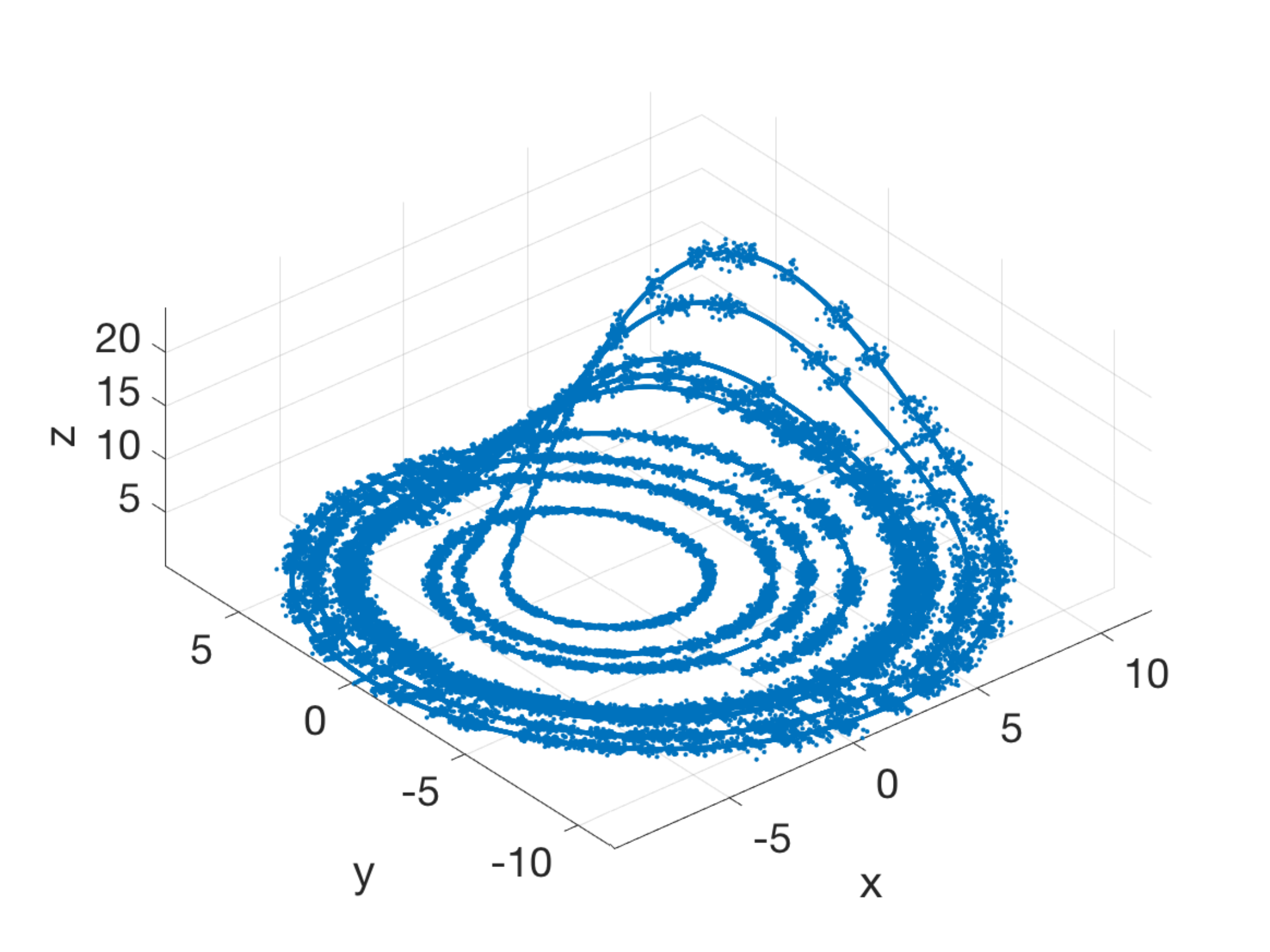}
\vspace{- 3cm}
\end{minipage}
\begin{minipage}[b]{0.48\linewidth}
\centering
  \begin{tabular}{ |c || c | c | c | }
    \hline
     & $ \dot{x_1}$  & $\dot{x_2}$ & $\dot{x_3}$ \\ \hline
    1 & 0 & 0  &  $0.1947$ \\ \hline
    $x_1$ & 0 & 1.0002 & 0 \\ \hline
    $x_2$ & -0.9922 & 0.2010 & 0 \\ \hline
    $x_3$ & -1.0203  & 0 & -5.7 \\ \hline
    $x_1^2$ & 0 & 0 & 0 \\ \hline
    $x_1x_2$ & 0 & 0 & 0.00 \\ \hline
    $x_1x_3$ & 0 & 0 & 0.9982 \\ \hline
    $x_2^2$ & 0 & 0 & 0 \\ \hline
    \vdots & \vdots &\vdots &\vdots \\ \hline
    $x_3^4$ & 0 & 0 &0 \\ \hline
    \hline
  \end{tabular}
\end{minipage}   
\caption{Left: R{\" o}ssler system \eqref{eqn:rossler} with 40\% corrupted data, Tfinal = 50, dt = 0.0005, hard-thres = 0.05, row-thres = 0.025. Right: the recovered coefficients. The model recovers the coefficients with 2\% error and detects exactly the locations of the outliers after 314 iterations.}
\label{fig:Rossler40}
\end{figure}
 
 In our final example, we apply our model to the following hyperchaos \cite{rossler1979equation}
 \begin{equation}
\begin{cases}
\frac{dx_1}{dt} &= -x_2-x_3\\
\frac{dx_2}{dt} & = x_1+0.25\, x_2 + x_4\\
\frac{dx_3}{dt} & =3 + x_1x_3\\
\frac{dx_4}{dt} &= -0.5\,x_3 + 0.05\,x_4
\end{cases}
\label{eqn:hyperchaos}
\end{equation}
As discussed in \cite{rossler1979equation}, the variable $x_3(t)$ serves to check the growth of the flow $(x_1,x_2,x_4)$ from time to time. We numerically solve the system \eqref{eqn:hyperchaos} using the fourth Runge-Kutta method with $dt = 0.001, Tfinal = 100$. The flow of the three-dimensional subspace $(x_1,x_2,x_4)$ is presented in Figure \ref{fig:hyperchaos} along with the remaining variable $x_3$ plot.
\begin{figure}
\centering
\includegraphics[width = 3.5 in]{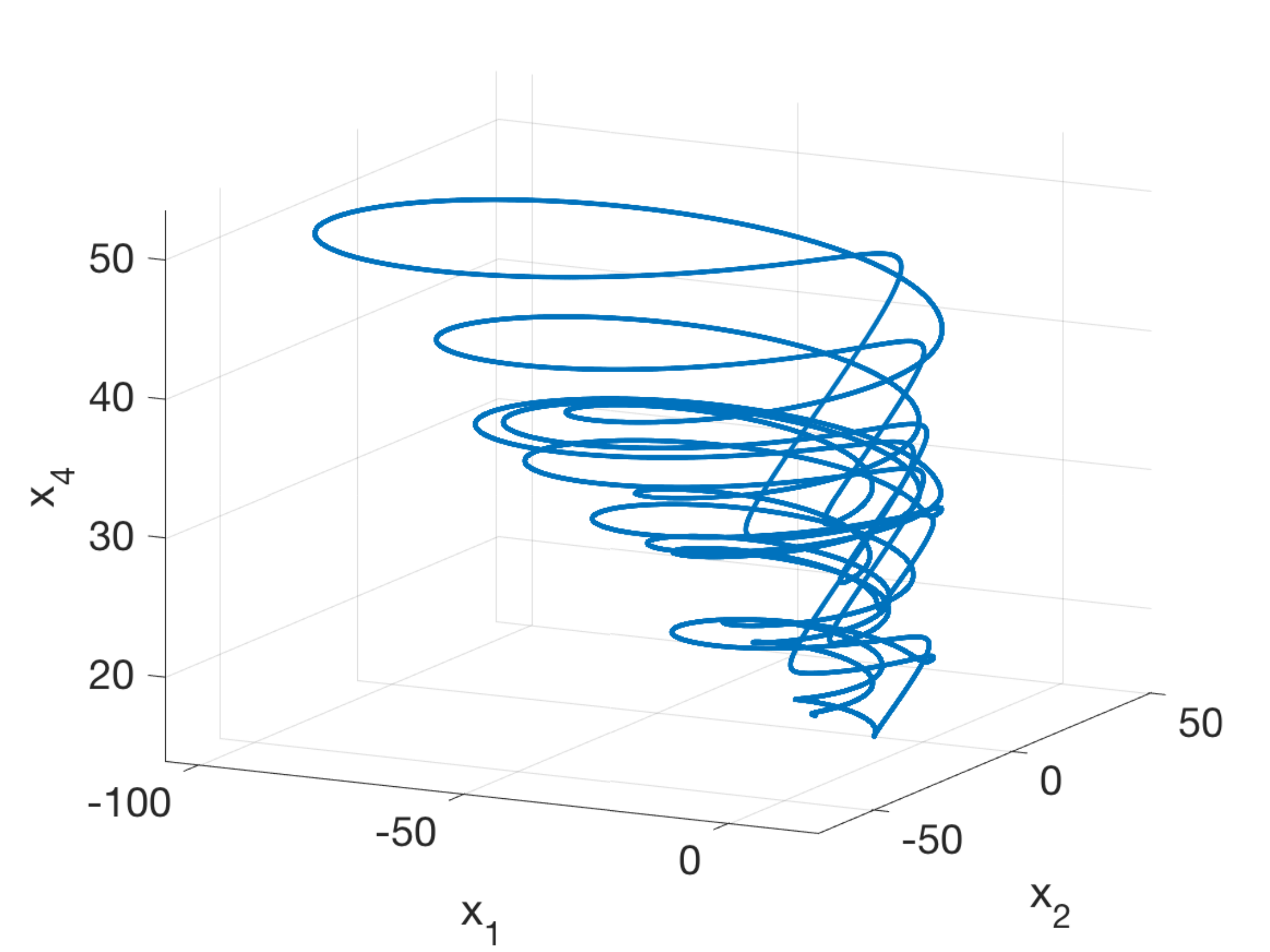}\quad
\includegraphics[width = 2.5 in]{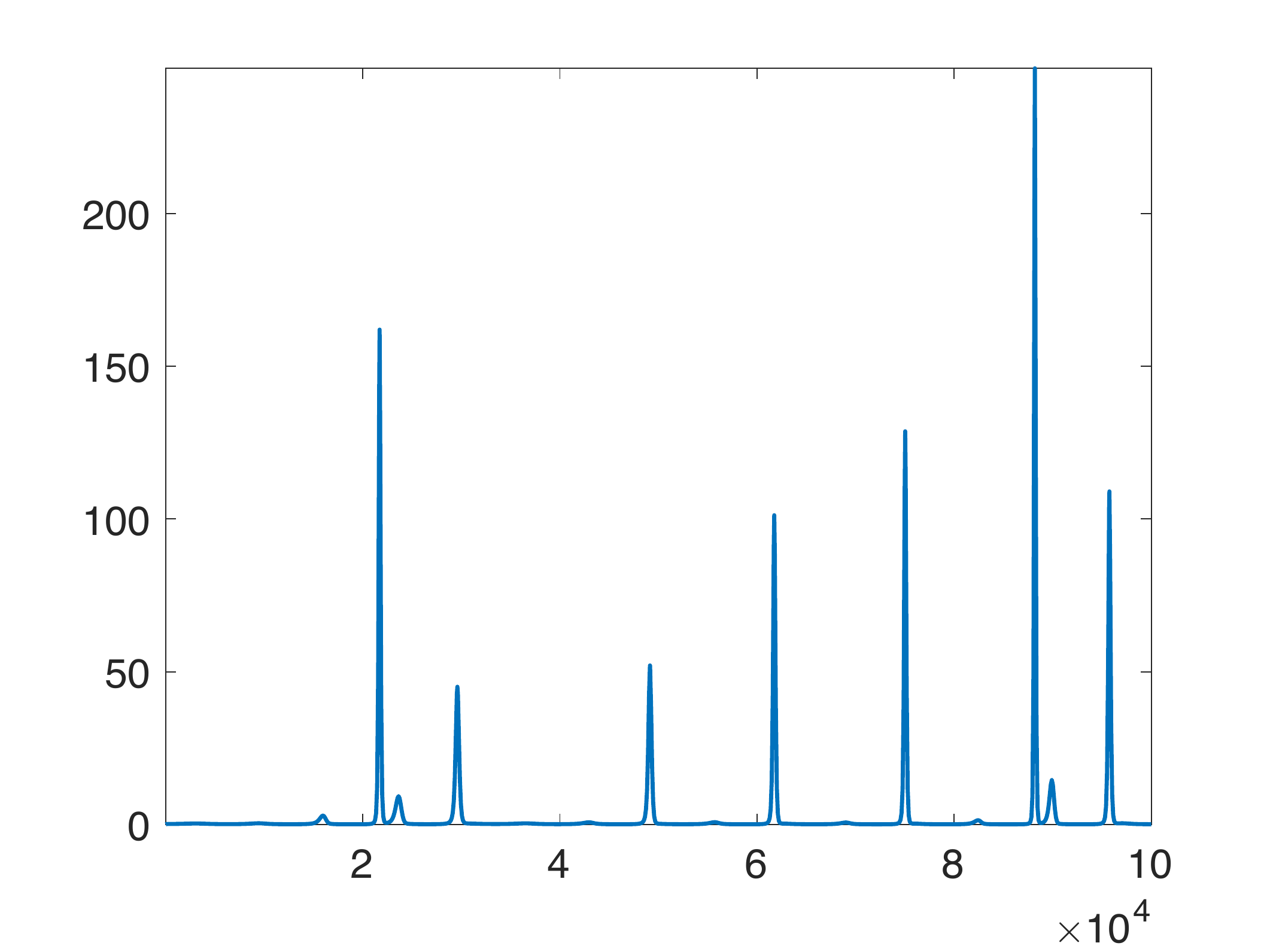}\quad
\caption{Left: Plot of trajectory flow in $(x_1,x_2,x_4)$ subspace, right: Plot of $x_3$ along time associated with the hyperchaos \eqref{eqn:hyperchaos}. Tfinal = 100, dt = 0.001.}
\label{fig:hyperchaos}
\end{figure}
Now we randomly assign the locations where the data is corrupted, and add Gaussian noise to the data at those corrupted intervals. The numerical result is shown in Figure \ref{fig:hyperchaos10}.
\begin{figure}
\begin{minipage}[b]{0.48\linewidth}
\centering
\includegraphics[width = 3. in]{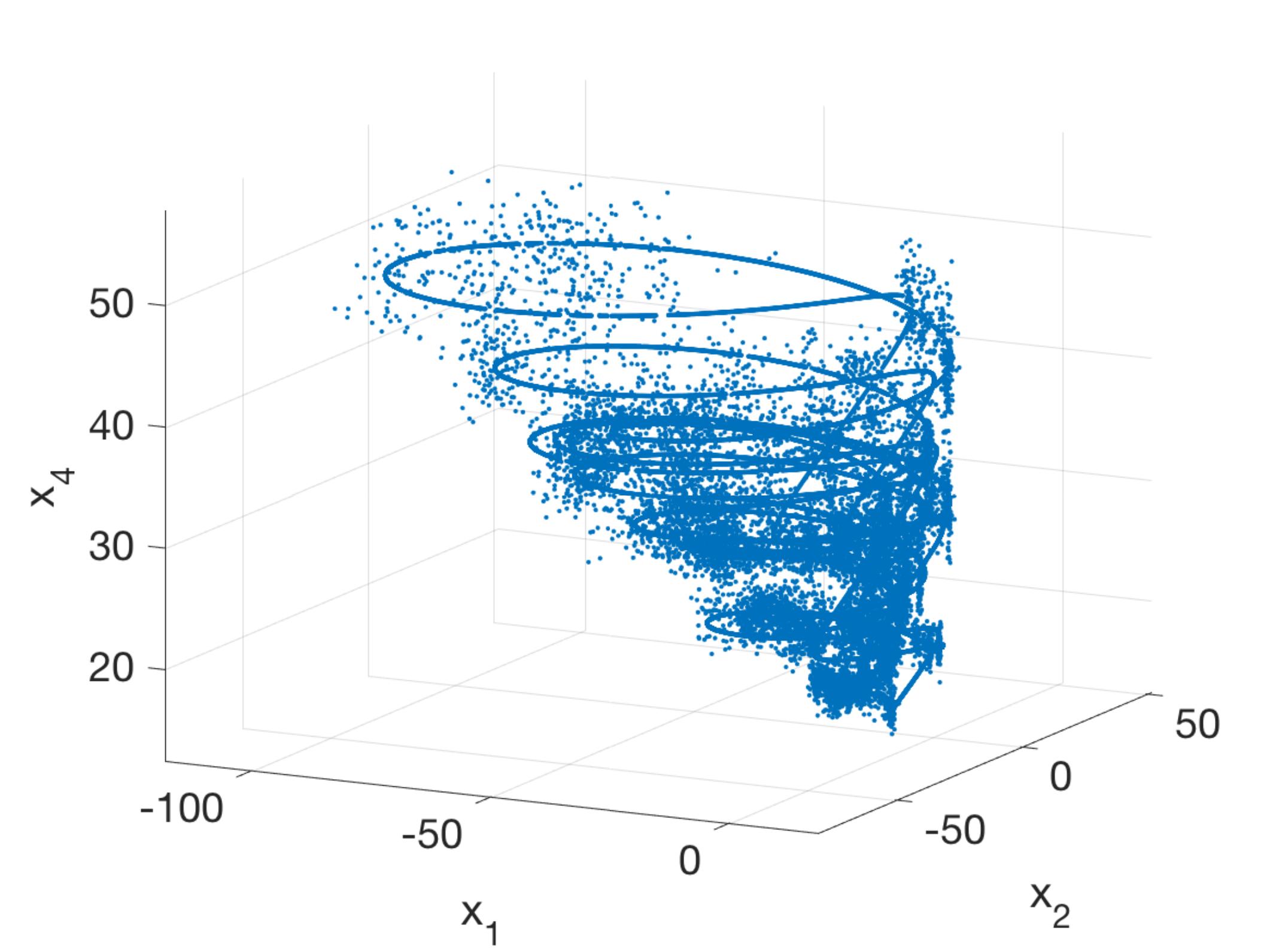}\\
\includegraphics[width = 2.5 in]{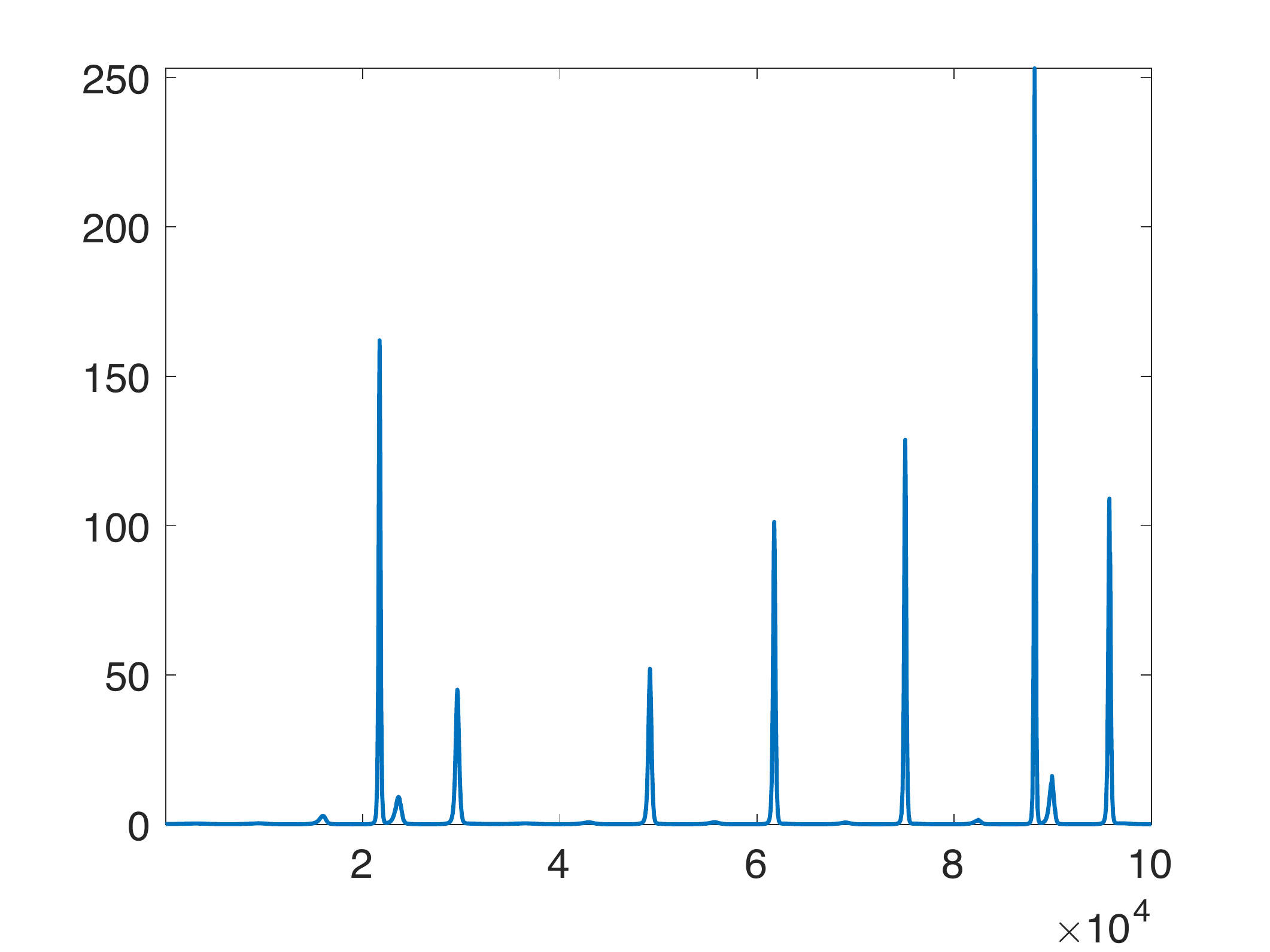}
\vspace{-3. cm}
\end{minipage}
\begin{minipage}[b]{0.48\linewidth}
\centering
  \begin{tabular}{ |c || c | c | c | c|}
    \hline
     & $ \dot{x_1}$  & $\dot{x_2}$ & $\dot{x_3}$ & $\dot{x_4}$ \\ \hline
    1 & 0 & 0  &  2.9998 & 0 \\ \hline
    $x_1$ & 0 & 0.99999 & 0 &0  \\ \hline
    $x_2$ & -0.99999 & 0.24999 & 0 & 0 \\ \hline
    $x_3$ & -0.99999  & 0 & 0& -0.49999 \\ \hline
    $x_4$ & 0 & 0.99999 & 0 & 0.04999\\ \hline
    $x_1^2$ & 0 & 0 & 0 & 0 \\ \hline
    $x_1x_2$ & 0 & 0 & 0 & 0 \\ \hline
    $x_1x_3$ & 0 & 0 & 0.99999 & 0 \\ \hline
    $x_1x_4$ & 0 & 0 & 0 & 0 \\ \hline
    \vdots & \vdots &\vdots &\vdots & \vdots \\ \hline
    $x_4^4$ & 0 & 0 &0 & 0\\ \hline
    \hline
  \end{tabular}
\end{minipage}   
\caption{Left: R{\" o}ssler hyperchaos system \eqref{eqn:rossler} with 10.25\% corrupted data, Tfinal = 100, dt = 0.001, hard-thres = 0.01, row-thres = 0.0125. Right: the recovered coefficients. The model recovers the coefficients and detect exactly the locations of the outliers after 201 iterations.}
\label{fig:hyperchaos10}
\end{figure}
\section{Conclusion and Discussion}
Using statistical properties of Lorenz-like chaotic systems and partial sparse recovery guarantees from compressive sensing, we provide conditions for recovering the governing equations from possibly highly corrupted measurement data. In addition, a stable numerical scheme is presented to recover the coefficients of the underlying equations in the space of multivariable polynomials with high accuracy and exactly identify the outliers, despite being in a regime of sensitivity to initial conditions. Our method might be useful for recovering the governing equations more generally, when the governing equations do not necessarily have a polynomial form. Explicitly, by doing our method "locally," we can approximate the Taylor series expansion to the governing equations (assuming they are sufficiently smooth), and then "piece together" the recovered polynomial Taylor expansions to approximate smooth governing equations more generally.
 In the future, we also would like to adjust the numerical scheme and theoretical guarantees to handle data with much higher level of noise in addition to outliers. We are also interested in extending theory and algorithm to settings where we observe snapshots of an observable of phase space (rather than the phase space measurement in its entirety) in the presence of outliers, as well as explore the sparsity structures in other high dimensional nonlinear functional spaces such as the space of Legendre polynomials \cite{rauhut2012sparse} and higher dimensional spaces. Finally, the constants in our theorem such that $m\geq C$ and $s\leq C'\,m^{1/(1+\eta)}$ are not explicit.  This follows because the constants are not explicit in the current theory for dynamical systems. So perhaps our work can motivate researchers in dynamical systems to make the associated constants more explicit.

\section*{Acknowledgements}
We would like to thank Afonso Bandeira and Stefan Steinerberger for helpful comments that improved this paper.  
R. Ward was partially supported by NSF CAREER grant $\#$1255631 and AFOSR YIP grant $\#$FA9550-13-1-0125. 

\bibliographystyle{alpha}

\bibliography{mainChaotic}

\end{document}